\documentclass[11pt,a4paper]{amsart}
\usepackage{amssymb,amsmath,amsthm,enumerate,graphicx,mathtools,booktabs,pinlabel,float}
\usepackage{stmaryrd,lscape,caption,subcaption}
\newtheorem{theorem}{Theorem}[section]
\usepackage{tikz}
\usetikzlibrary{shapes,shadows,arrows,trees}
\usepackage{epstopdf}
\usepackage{inputenc}
\newtheorem{prop}[theorem]{Proposition}
\newtheorem{lemma}[theorem]{Lemma}
\newtheorem{cor}[theorem]{Corollary}
\newtheorem{cor*}{Corollary}
\theoremstyle{definition}
\newtheorem{defn}[theorem]{Definition}
\newtheorem{rem}[theorem]{Remark}
\newtheorem{exmp}[theorem]{Example}
\newcommand{\m}{\text{Mod}(S_g)}
\newcommand{\Mod}{\text{Mod}}
\newcommand{\Homeo}{\text{Homeo}}
\newcommand{\p}{\text{PSL}_2(\mathbb{R})}
\newcommand{\h}{\mathbb{H}}
\newcommand{\lcm}{\text{lcm}}

\newcommand{\lb}{\llbracket}
\newcommand{\rb}{\rrbracket}
\renewcommand{\P}{\mathcal{P}}
\renewcommand{\O}{\mathcal{O}}
\newcommand{\stab}{\text{Stab}}

\begin{document}

\title[Commuting conjugates of finite-order mapping classes]{Commuting conjugates of \\finite-order mapping classes}

%    Information for author 1
\author{Neeraj K. Dhanwani}
\address{Department of Mathematics\\
Indian Institute of Science Education and Research Bhopal\\
Bhopal Bypass Road, Bhauri \\
Bhopal 462 066, Madhya Pradesh\\
India}
\email{nkd9335@iiserb.ac.in}

%    Information for author 2
\author{Kashyap Rajeevsarathy}
\address{Department of Mathematics\\
Indian Institute of Science Education and Research Bhopal\\
Bhopal Bypass Road, Bhauri \\
Bhopal 462 066, Madhya Pradesh\\
India}
\email{kashyap@iiserb.ac.in}
\urladdr{https://home.iiserb.ac.in/$_{\widetilde{\phantom{n}}}$kashyap/}



\maketitle

\begin{abstract}
 Let $\m$ be the mapping class group of the closed orientable surface $S_g$ of genus $g\geq 2$. In this paper, we derive necessary and sufficient conditions for two finite-order mapping classes to have commuting conjugates in $\m$. As an application of this result,  we show that any finite-order mapping class, whose corresponding orbifold is not a sphere, has a conjugate that lifts under any finite-sheeted cover of $S_g$. Furthermore, we show that any torsion element in the centralizer of an irreducible finite order mapping class is of order at most $2$. We also obtain conditions for the primitivity of a finite-order mapping class. Finally, we describe a procedure for determining the explicit hyperbolic structures that realize two-generator finite abelian groups of $\m$ as isometry groups. 
\end{abstract}

\section{introduction}
\label{sec:intro}
Let $S_g$ denote closed orientable surface of genus $g \geq 0$, and let $\Mod(S_g)$ denote the mapping class group of $S_g$. Given two finite-order mapping classes in $\m$, for $g \geq 2$, a natural question that arises is whether there exist representatives of their respective conjugacy classes that commute in $\m$. (When two finite-order mapping classes satisfy this condition, we say that they \textit{weakly commute}.) While finite abelian groups and their conjugacy classes in $\m$ have been widely studied~\cite{BW,H1,M2}, our pursuit can be motivated with the following example. Consider the six involutions in $\Mod(S_8)$ shown in Figure~\ref{fig:k4_s4} below, where each involution is realized as a $\pi$-rotation about an axis (passing through the origin) under a suitable isometric embedding $S_8 \hookrightarrow \mathbb{R}^3$.
	\begin{figure}[H]
		\labellist
		\small
		\pinlabel $\pi$ at 70 72
		\pinlabel $\pi$ at 147 72
		\pinlabel $\pi$ at 300 72
		\pinlabel $\pi$ at 337 190
		\pinlabel $x$ at 40 47
		\pinlabel $y$ at 183 50
		\pinlabel $x$ at 267 50
		\pinlabel $z$ at 325 233
		\pinlabel $x$ at 475 35
		\pinlabel $y=x$ at 575 24
		\endlabellist
		\centering
		\includegraphics[width=60ex]{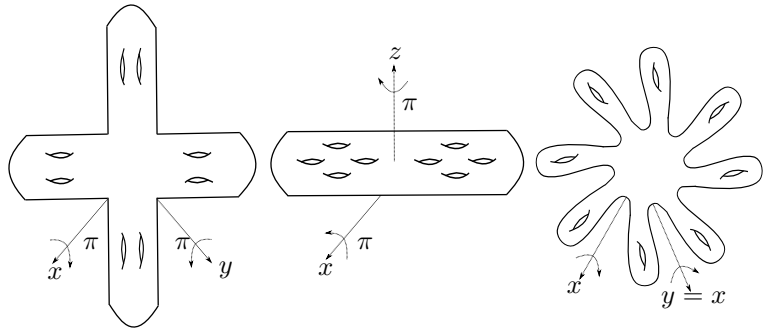}
		\caption{Six conjugate involutions in $\Mod(S_8)$}.
		\label{fig:k4_s4}
	\end{figure}
\noindent Though all of these involutions are conjugate in $\Mod(S_g)$, note that each of the two pairs of involutions indicated in the first two subfigures clearly generate distinct subgroups of $\Mod(S_8)$ isomorphic to $\mathbb{Z}_2 \oplus \mathbb{Z}_2$, while the pair of involutions appearing in the third figure can be shown to generate a subgroup isomorphic to $D_{16}$. 

As the main result of this paper (see Theorem~\ref{main theorem}), in Section~\ref{sec:main}, we derive necessary and sufficient conditions under which two finite-order mapping classes will have commuting conjugates in $\m$. We appeal to Thurston's orbifold theory~\cite{WT}, and the classical theory~\cite{H1,SK1,M1} of group actions on surfaces for proving this result. A key ingredient in our proof is understanding the factors that determine whether a given $\mathbb{Z}_n$-action on $S_g$ induces a  $\mathbb{Z}_n$-action on the quotient orbifold of another cyclic action, and also analyzing the properties of such an induced action.  In this connection, we also provide an abstract tuple of integers called an ``abelian data set" which corresponds to a two-generator finite abelian subgroup up to a notion of equivalence that we call ``weak conjugacy", which, as the term suggests, is weaker than conjugacy (see Section~\ref{sec:main}). 

Let $F \in \m$ be of order $n$. By the Nielsen-Kerckhoff theorem~\cite{SK,JN}, $F$ has a representative $\widetilde{F} \in \Homeo^+(S_g)$ such that $\widetilde{F}^n = 1$. We call the quotient orbifold $S_g / \langle \widetilde{F} \rangle$ the  \textit{corresponding orbifold} for $F$. For an $m$-sheeted cover $p : S_{m(g-1)+1} \to S_g$, let $\text{LMod}_{p}(S_g)$ denote the subgroup of $\Mod(S_g)$ of liftable mapping classes under $p$. As a first application of our main result, in Section~\ref{sec:appl}, we derive conditions under which a finite-order mapping classes weakly commute with mapping classes represented by generators of certain free cyclic actions on $S_g$ (see Corollary~\ref{cor:free_comm}). A direct consequence of this result is the following:

\begin{cor*}
Let $p:S_{m(g-1)+1}\to S_g$ be an $m$-sheeted cover whose deck transformation group is $\mathbb{Z}_m$. Let $F \in \Mod(S_g)$ be a finite-order mapping class whose corresponding orbifold is not a sphere. Then the conjugacy class of $F$ has a representative $F'$ such that $F' \in \text{LMod}_{p}(S_g)$.
\end{cor*}

\noindent We also derive an analog of this corollary for certain finite-order mapping classes whose corresponding orbifolds are spheres (see Corollary~\ref{cor:lift_sphere}). It is known~\cite{H1,AW} that an $F \in \m$ of finite order with $|F| > 2g+1$ is primitive. Using our theory, we give conditions that determine the primitivity (see Theorem~\ref{primitivity}) of an arbitrary finite-order mapping class. These conditions further lead to a characterization of the primitivity of certain surface rotations.

\begin{cor*}
Let $F \in \m$ be a finite-order mapping class. 
	\begin{enumerate}[(i)]
		\item If $|F| = g-1$ and $F$ is represented by the generator of a free action, then a nontrivial root $G$ of $F$ exists if, and, only if $2 \nmid (g-1)$. Moreover, $G$ has degree $2$.
		\item If $6 \mid g$ and $F$ is represented by a rotation of order $g$, then $F$ is primitive.
	\end{enumerate}
\end{cor*}

It is known~\cite{G1} that a finite-order mapping class is irreducible if, and only if, its corresponding orbifold is a sphere with $3$ cone points. Following the nomenclature from ~\cite{PKS}, we say an irreducible order $n$ mapping class is of \textit{Type 1} if its corresponding orbifold has a cone point of order $n$, otherwise we say such a mapping class is of \textit{Type 2}. In this connection, we prove the following: 
 
\begin{cor*}
Suppose that a finite abelian subgroup $A$ of $\m$ contains an irreducible finite-order mapping class $F$. 
\begin{enumerate}[(i)]
\item If $F$ is of Type 2, then $A = \langle F \rangle$. 
\item If $F$ is of Type 1, then either $A = \langle G \rangle$, where $G$ is a root of $F$, or $A \cong \mathbb{Z}_2 \oplus \mathbb{Z}_{2g+2}.$ 
\end{enumerate}
\end{cor*}

 Let $c$ be a simple closed curve in $S_g$ for $g \geq 2$, and let $t_c \in \Mod(S_g)$ denote the left-handed Dehn twist about $c$. Let $F \in \m$ be either a root of $t_c$ of degree $n$, or an order-$n$ mapping class that preserves the isotopy class of $c$. Then we may assume up to isotopy that $F(c)=c$, and that $F$ preserves a closed annular neighborhood $N$ of $c$. Further, it is known~\cite{MK1,KR1,PKS} that $F$ induces an order-$n$ map $\widehat{F}_c$ on the surface obtained by capping off the components of $\overline{S_g \setminus N}$. As another application of our main result, we obtain the following characterization of weak commutativity of finite-order mapping classes with roots of Dehn twists about nonseparating curves.

\begin{cor*}
Let $F \in \m$ be a root of $t_c$, where $c$ is nonseparating, and $G \in \m$ be of finite order. Then $F$ and $G$ have commuting conjugates if, and only if $G(c) = c$, and $\widehat{F}_{c}$ and $\widehat{G}_{c}$ have commuting conjugates. In particular, if $\widehat{F}_c$ is primitive, then $F$ and $G$ cannot commute in $\m$. 
\end{cor*}
\noindent We also state an analog of this result (see Corollary~\ref{cor:weak_comm_sep_root}) for the roots of Dehn twists about separating curves. 

Given a weak conjugacy class of a two-generator finite abelian group (encoded by an abelian data set), in Section~\ref{sec:factor_gens}, we provide an algorithm for determining the conjugacy classes of its generators. We indicate how this algorithm, along with theory developed in~\cite{PKS}, leads to a procedure for determining the explicit hyperbolic structures that realize a two-generator abelian subgroup as a group of isometries. Finally, we classify the weak conjugacy classes of two-generator finite abelian subgroups of $\Mod(S_3)$. We conclude the paper by providing some non-trivial geometric realizations of some of these subgroups. 

\section{Preliminaries}
A \textit{Fuchsian group}~\cite{SK1} $\Gamma$ is a discrete subgroup of $\text{Isom}^+(\h) = \p$. If $\h / \Gamma$ is a compact orbifold, then $\Gamma$ has a presentation of the form
\begin{gather*}
\left\langle \alpha_1,\beta_1,\dots,\alpha_{g_0},\beta_{g_0}, \xi_1,\dots,\xi_{\ell} \, |\,  
\xi_1^{n_1}=\dots=\xi_\ell^{n_{\ell}}=\prod_{i=1}^{\ell} \xi_i \prod_{i=1}^{g}[\alpha_i,\beta_i] =1\right\rangle.
\end{gather*}
 We represent $\Gamma$ by a tuple $(g_0;n_1,n_2,\dots,n_{\ell})$ which is called its \textit{signature}, and we write $$\Gamma(g_0;n_1,n_2,\dots,n_{\ell}) := \Gamma.$$ Let $\Homeo^+(S_g)$ denote the group of orientation-preserving homeomorphisms on $S_g$. Given a finite group $H < \Homeo^+(S_g)$, a faithful properly discontinuous $H$-action on $S_g$ induces a branched covering $$S_g \to \mathcal{O}_H := S_g/H,$$ which has $\ell$ branched points (or cone points) $x_1,\ldots ,x_{\ell}$ in the quotient orbifold $\mathcal{O}_H \approx S_{g_0}$ of orders $n_1, \ldots ,n_{\ell}$, respectively. Thus, $\O_H$ has a signature given by 
$$\Gamma(\O_H) := (g_0;n_1,n_2,\dots,n_{\ell}),$$ and its orbifold fundamental group is given by
$$\pi_1^{orb}(\O_H) := \Gamma(g_0;n_1,n_2,\dots,n_{\ell}).$$
From orbifold covering space theory, the orbifold covering map $S_g \to \O_H$ corresponds to
an exact sequence
$$1 \rightarrow \pi_1(S_g) \rightarrow \pi_1^{orb}(\O_H) \xrightarrow{\phi_H}  H \rightarrow 1.$$
This leads us to the following result~\cite{H1} due to Harvey.
 
\begin{lemma}
\label{Harvey Condition} 
A finite group $H$ acts faithfully on $S_g$ with $\Gamma(\O_H) = (g_0;n_1,\dots,n_{\ell})$ if, and only if, it satisfies the following two conditions: 
\begin{enumerate}[(i)]
\item $\displaystyle \frac{2g-2}{|H|}=2g_0-2+\sum_{i=1}^{\ell}\left(1-\frac{1}{n_i}\right)$, and 
\item  there exists a surjective homomorphism $\phi_H:\pi_1^{orb}(\O_H) \to H$ such that preserves the orders of all torsion elements of $\Gamma$.
\end{enumerate}
\end{lemma}
 
For $g \geq 1$, let $H = \langle F \rangle$ be a finite cyclic subgroup of $\Mod(S_g)$ of order $n$. By the Nielsen-Kerckhoff theorem~\cite{SK,JN}, we may also regard $H$ as a finite cyclic subgroup of $\Homeo^+(S_g)$ generated by an $\widetilde{F}$ of order $n$. We call $\widetilde{F}$ a \textit{standard representative} of the mapping class $F$. For notational simplicity, we will also denote the standard representative $\widetilde{F}$ by $F$. We refer to both $F$ and the group it generates, interchangeably, as a \textit{$\mathbb{Z}_n$-action on $S_g$}.  Moreover, $F$ corresponds to an orbifold $\O_H \approx S_g/H$ (called the \textit{corresponding orbifold}), where for each $i$, the cone point $x_i$ lifts to an orbit of size $n/n_i$ on $S_g$. The local rotation induced by $F$ around the points in the orbit is given by $2 \pi c_i^{-1}/n_i$, where $c_i c_i^{-1} \equiv 1 \pmod{n_i}$. We denote a typical point in $\mathcal{O}_H$ by $[x]$, where $x$ is a lift of $[x]$ under the branched cover $S_g \to \O_H$. We see that each cone point $[x] \in \mathcal{O}_H$ corresponds to a unique pair in the multiset $\{(c_1,n_1),\ldots,(c_{\ell},n_{\ell})\}$, which we denote by $(c_{x},n_{x})$. So, we define 
$$\P_{[x]} := \begin{cases}
                      (c_{x},n_{x}), & \text{if $[x]$ is a cone point of $\O_H$, and} \\
                      (0,1), & \text{otherwise.}
                    \end{cases}$$
\begin{defn}\label{defn:data_set}

\noindent We will now define a tuple of integers that will encode the conjugacy class of a $Z_n$-action on $S_g$. 

A \textit{data set of degree $n$} is a tuple
$$
D = (n,g_0, r; (c_1,n_1),\ldots, (c_{\ell},n_{\ell})),
$$
where $n\geq 2$, $g_0 \geq 0$, and $0 \leq r \leq n-1$ are integers, and each $c_i$ is a residue class modulo $n_i$ such that:
\begin{enumerate}[(i)]
\item $r > 0$ if, and only if $\ell = 0$, and when $r >0$, we have $\gcd(r,n) = 1$, 
\item each $n_i\mid n$,
\item for each $i$, $\gcd(c_i,n_i) = 1$, 
\item $\lcm(n_1,\ldots \widehat{n_i}, \ldots,n_{\ell}) = N$, for $1 \leq i \leq r$, where $N = n$, if $g_0 = 0$,  and
\item $\displaystyle \sum_{j=1}^{\ell} \frac{n}{n_j}c_j \equiv 0\pmod{n}$.
\end{enumerate}
The number $g$ determined by the Riemann-Hurwitz equation
\begin{equation*}\label{eqn:riemann_hurwitz}
\frac{2-2g}{n} = 2-2g_0 + \sum_{j=1}^{\ell} \left(\frac{1}{n_j} - 1 \right) 
\end{equation*}
is called the \emph{genus} of the data set.
\end{defn}

\noindent The following lemma is a consequence of~\cite[Theorem 3.8]{KP} and the results in~\cite{H1}.

\begin{lemma}\label{prop:ds-action}
For $g \geq 1$ and $n \geq 2$, data sets of degree $n$ and genus $g$ correspond to conjugacy classes of $\mathbb{Z}_n$-actions on $S_g$. 
\end{lemma}

\noindent  The quantity $r$ associated with a data set $D$ will be non-zero if, and only if, $D$ represents a free rotation of $S_g$ by $2\pi r/n$. We will avoid writing $r$ in the notation of a data set,
whenever $r = 0$. From here on, we will use data sets to denote the conjugacy classes of cyclic actions on $S_g$.  Given a finite-order mapping class $F$, we define the data set associated with its conjugacy class by $D_F$. Further, for convenience of notation, we also write the data set $D$ as
$$D = (n,g_0,r; ((d_1,m_1),\alpha_1),\ldots,((d_r,m_r),\alpha_r)),$$
where $(d_i,m_i)$ are the distinct pairs in the multiset $S = \{(c_1,n_1),\ldots,(c_{\ell},n_{\ell})\}$, and the $\alpha_i$ denote the multiplicity of the pair $(d_i,m_i)$ in $S$. 

\section{Induced automorphisms on quotient orbifolds}
\label{sec:ind_aut}
Consider a finite group $H < \Homeo^+(S_g)$, and a subgroup $H' \lhd H$. Then it is known~\cite{TWT} that the actions of $H$ and $H'$ on $S_g$ induces an action of $H/H'$ on $\O_{H'}$. In this section, we analyze this induced action for the case when $H$ is a two-generator finite abelian group, and $H'$ is one of its cyclic factor subgroups. We will derive several properties of these induced actions, which will form the core of the theory that we develop in this paper.

\begin{defn}
Let $H < \Homeo^+(S_g)$  be a finite cyclic group. We say a $\bar{F} \in \Homeo^+(\O_H)$  is an \textit{automorphism of $\O_H$} if $\bar{F}([x]) = [y]$, for some $[x],[y] \in \O_H$, then 
$\P_{[x]} = \P_{[y]}$. 
\end{defn}

\noindent We denote the group of automorphisms of $\O_H$ by $\text{Aut}(\O_H)$. We derive three technical lemmas, which give necessary conditions under which a given orbifold automorphism is induced by a finite-order map. These lemmas will be used extensively in subsequent sections. 

\begin{lemma}{\label{commuting quationt} }
Let $G,F \in \Homeo^+(S_g)$ be commuting maps of order $m,n$, respectively, and let $H = \langle F\rangle$. Then:
\begin{enumerate}[(i)]
\item  $G$ induces a $\bar{G}\in \Homeo^+(\O_H)$ such that $$\O_H/\langle \bar{G}\rangle = S_g/\langle F,G\rangle,$$ 
\item $|\bar{G}| \mid |G|$, and 
\item $|\bar{G}| < m$ if, and only if, $F^{l}=G^k$, for some $0< k< m$ and $0< l<n$.
\end{enumerate}
\end{lemma}
\begin{proof}
Defining $\bar{G}[x]=[G(x)]$, for $[x] \in S_g/\langle F \rangle$, we see the (i) follows immediately. The assertion in (ii) follows from the fact that
	$$\bar{G}^{m}([x])=[G^{m}(x)]=[x], \text{ for } [x] \in S_g/\langle F \rangle.$$
To prove (iii), we first assume that $t:=|\bar{G}| < m$. Suppose we assume on the contrary that $F^{l}\neq G^k$, for any $1 \leq l < n$ and $1 \leq k < m$. Then
	$$\bar{G}^t([x])=[x] \Leftrightarrow [G^t(x)]=[x],$$ for all $[x] \in \O_H$. Thus, for each $[x] \in \O_H$, there exists $1  \leq l_x\leq n$ such that $G^t  F^{l_x}(y)=y,$ for all $y \in S_g$ in the preimage of $[x]$ under the branched cover $S_g \to \O_H$. 
	If $t < m$, then for each $l_x$, $G^t  F^{l_x}$ is a non-trivial homeomorphism, which shows that every point of $S_g$ is fixed by some element of the abelian group $\langle F,G \rangle$ of order $mn$, which is impossible. The converse follows directly from the definition of $\bar{G}$.

\end{proof}

\noindent We call the map $\bar{G}$ in Lemma~\ref{commuting quationt} the \textit{induced map on $\O_{\langle F \rangle}$ by G}. For an action of a group $G$ on a set $X$, we denote the stabilizer of a point $x \in X$ by $\stab_G(x)$. We will also need the following well known result~\cite[Proposition 3.1]{RM} from the theory of finite group actions on surfaces. 

\begin{lemma}
\label{lem:stab_cyc}
Let $H < \Homeo^+(S_g)$ be finite. Then $\stab_H{(x)}$ is a cyclic group, for every $x \in S_g$.
\end{lemma}

\begin{lemma}{\label{cone condition}}
	Let $F,G \in \Homeo^+(S_g)$ be of orders $n,m$, respectively, and let $\bar{F} \in \Homeo^+(\O_{\langle G \rangle})$ be induced by $F$ as in Lemma~\ref{commuting quationt}. Suppose that $FG = GF$, and $F^{p}\neq G^q$, for any $1<q< n$ and $1<p<m$. If for some $x \in S_g$, $G^k(x) = x$ and $\bar{F}^l([x])=[x]$, for some $1 \leq k < m$ and $1 \leq l < n$, then $$|\bar{F^l}|=ba,$$ where $\gcd(b,m)=1$ and $a \mid \frac{m}{|G^k|}$. 
	\end{lemma}
	\begin{proof}
		It suffices to establish the result for the case when $|G^k| = m$, that is, for $k=1$. Suppose we assume on the contrary that $|\bar{F}^l|=b$, where $\gcd(m,b)=\alpha\neq 1.$ Then there exists $1\leq t\leq m $ such that $G^t  F^l(x)=x$. Thus, we have that $G^{\frac{m}{\alpha}},F^{\frac{lb}{\alpha}} \in \stab_A(x),$  where $A = \langle F, G \rangle$. Since $\stab_A(x)$ is cyclic and $|G^{\frac{m}{\alpha}}| = |F^{\frac{lb}{\alpha}}| = \alpha$, we have $G^{\frac{m}{\alpha}}\in \langle F^{\frac{lb}{\alpha}}\rangle,$ which is impossible. Hence, our assertion follows.
	\end{proof}
	
\begin{lemma}
	\label{ind condition}
	Let $G,F \in \Homeo^+(S_g)$ be commuting homeomorphisms of orders $m,n$, respectively. Let $\bar{F}$ be the induced map on $S_g/\langle G \rangle$ as in Lemma~\ref{commuting quationt}. Then:
	\begin{enumerate}[(i)] 
		\item For $[x], [y] \in \mathcal{O}_{\langle G \rangle}$, if $\bar{F}([x]) = ([y])$, then $\P_x = \P_y$.
		\item For each orbit $O$ of size $|F|$ induced by the action of $\langle \bar{F} \rangle$ on $\O_{\langle G \rangle}$, there exists a point $[x(O)] \in \O_{\langle \bar{G} \rangle}$ such that 
		$\P_{[x(O)]} = \P_{[y]}$, where $[y] \in O$.
		\item Let $F$ have $\beta$ fixed points in $S_g$. If $\bar{\beta}$ denotes the number of fixed points of $\bar{F}$, then 
		$$\left \lceil \frac{\beta}{m} \right \rceil \leq \bar{\beta} \leq \left\lfloor \frac{(m-1)(2g-2+2n)}{m(n-1)} \right \rfloor 
		+ \left \lceil \frac{\beta}{m} \right\rceil.$$
		
	\end{enumerate}
\end{lemma}
\begin{proof}
\begin{enumerate}[(i)]
\item Suppose that $\bar{F}([x]) = [y]$. Then there exists $x',y' \in S_g$ in the pre-images of $[x],[y]$ (under the branched cover) such that $F(x') = y'$. Then $$G^{m/n_x}(y')=G^{m/n_x}(F(x'))=F(G^{m/n_x}(x'))=F(x')=y',$$ where $\P_{[x]} = (c_x,n_x)$. By a similar argument, we can show that $G^{n_y}(x')=x'$, and so it follows that $n_x = n_y$. 
	
	To show that $c_x = c_y$, it now suffices to show that if $\bar{F}([x]) = [y]$, where $n_x = n_y = m$, then $c_x=c_y$. Without loss of generality, we assume that $c_y=1$. Now, there exists an $G$-invariant disk $D_2$ around $y$ that $G$ rotates by $2\pi/m$, and there exists a $G$-invariant disk $D_1$ around $x$ that $G$ rotates by $2\pi c_x^{-1}/m$. So, we must have $FGF^{-1}=G^{c_x}$, which is impossible, as $F$ and $G$ commute.
	
\item Suppose that $F$ has $m$ fixed points $\{x_1,\dots x_m\}$ that form an orbit under the action of $G$ on $S_g$. Then, it is clear that $\bar{F}([x_1])=[x_1]$, from which the assertion follows.
	
\item If $F(x)=x$, then by definition, $\bar{F}([x])=[x]$, and so we have $F(G^i(x)) = G^i(x)$, for each $i$. If $F$ has $\beta$ fixed points, then there exist atleast $\frac{\beta}{m}$ distinct orbits which contain points fixed by $F$. Hence, the lower bound follows. 
	
	To show the upper bound, we observe that if $\bar{F}([x])=[x]$, then by definition, there exist $0\leq i\leq m-1$ such that $G^iF(x)=x$. When $i\neq 0$, by a direct application of the Riemann-Hurwitz equation, it follows that $\left\lfloor \frac{(2g-2+2n)}{(n-1)} \right \rfloor$ is maximum number cone points of order $n$ in $\O_{\langle G^iF \rangle}$, which completes the argument.
	\end{enumerate}
\end{proof}

\noindent The necessary conditions that appear in lemmas above, under which a given orbifold automorphism is induced, are summarized in the following two definitions.

\begin{defn} 
Let $F,G \in \Homeo^+(S_g)$ be of orders $n$ and $m$ respectively, and let $H = \langle G \rangle$. We say a map $\bar{F} \in \text{Aut}(\O_H)$ satisfies the \textit{induced map property (IMP) with respect to $(F,G)$}, if the following conditions hold. 
\begin{enumerate}[(i)]
\item  For $[x], [y] \in \mathcal{O}_H$, if $\bar{F}([x]) = ([y])$, we have $\P_x = \P_y$.
\item For each orbit $O$ of size $|F|$ induced by the action of $\langle \bar{F} \rangle$ on $\O_{\langle G \rangle}$, there exists a point $[x(O)] \in \O_{\langle \bar{G} \rangle}$ such that 
		$\P_{[x(O)]} = \P_{[y]}$, where $[y] \in O$.
\item Let $F$ have $\beta$ fixed points in $S_g$. If $\bar{\beta}$ denotes the number of fixed points of $\bar{F}$, then 
		$$\left \lceil \frac{\beta}{m} \right \rceil \leq \bar{\beta} \leq \left\lfloor \frac{(m-1)(2g-2+2n)}{m(n-1)} \right \rfloor 
		+ \left \lceil \frac{\beta}{m} \right\rceil.$$
\item If $[x]$ is a cone point of order $n'$ in $\O_H$, then $\bar{F}^l([x])=[x]$, only if $ \mid \bar{F}^l\mid=ba$, where $\gcd(b,m)=1$ and $a\mid \frac{m}{n'}$.

\end{enumerate}
\end{defn}

\begin{defn}
\label{defn:ess_pair}
Let $F,G \in \Homeo^+(S_g)$ be finite-order maps with $D_F =  (n,g_1,r_1;((c_1,n_1),\alpha_1),\dots,((c_r,n_r),\alpha_r))$ and $D_G = (m,g_2,r_2;((d_1,m_1),\beta_1),\linebreak\dots,((d_k,m_k),\beta_k))$, where $m \mid n$. Then $(G,F)$ are said to form an \textit{essential pair} if the following three conditions hold. 
\begin{enumerate}[(i)]
		\item There exists a $\widetilde{F} \in \Homeo^+(S_{g_2})$ with $D_{\widetilde{F}} = (n,g_0,r^o_1;(c^o_1,n_1^o),\dots,(c^o_s,n_s^o))$ on $S_{g_2}$ which induces an $\bar{F} \in \text{Aut}({O}_{\langle G \rangle})$ that satisfies the IMP with respect to $(F,G)$.
		\item There exists a $\widetilde{G} \in \Homeo^+(S_{g_1})$ with $D_{\widetilde{G}} =(m,g_0,r^o_2;(d_1^o,m_1^o),\dots,(d_t^o,m_t^o))$, which induces a $\bar{G} \in \text{Aut}(\mathcal{O}_{\langle F \rangle})$  that satisfies the IMP with respect to $(G,F)$.
		\item $\Gamma(\mathcal{O}_{\langle G \rangle}/\langle\bar{F}\rangle)=\Gamma(\mathcal{O}_{\langle F \rangle}/\langle\bar{G}\rangle)$.
\end{enumerate}
The number $mn$ (written as $m \cdot n$) is called the \textit{order} of the essential pair $(G,F)$.
\end{defn}

\begin{exmp}\
	Let $F,G\in \Homeo^+(S_7)$ with $D_F=D_G=(6,2,1;)$. Then $(G,F)$ is an essential pair of order $6\cdot6$, as $F,G$ induce $\bar{F},\bar{G}\in \Homeo^+(S_2)$ (resp.) with $D_{\bar{F}}=D_{\bar{G}}=(6,0;((1,2),2),(1,3),(2,3))$, and $\Gamma(\mathcal{O}_{\langle G \rangle}/\langle\bar{F}\rangle)=\Gamma(\mathcal{O}_{\langle F \rangle}/\langle\bar{G}\rangle) = (0;2,2,3,3)$.	
\end{exmp}

\noindent Given a quotient orbifold $\O_H$, where $H = \langle F \rangle$, we now state a set of necessary conditions (as we will show later in Theorem~\ref{main theorem}) for a given $\bar{G} \in \text{Aut}(\O_H)$ to be induced by a finite-order map $G$ such that $\langle G, F \rangle$ forms a two-generator abelian group. 

\begin{defn}
\label{defn:weak_ab_pair}
For finite-order maps $F,G \in \Homeo^+(S_g)$, let $(G,F)$ form an essential pair of order $m\cdot n$ as in Definition~\ref{defn:ess_pair}. Then $(G,F)$ is said to be a \textit{weakly abelian pair of order $m \cdot n$} if the following conditions hold.
\begin{enumerate}[(i)]
		\item If $\Gamma(\mathcal{O}_{\langle G \rangle}/\langle\bar{F}\rangle)=\Gamma(\mathcal{O}_{\langle F \rangle}/\langle\bar{G}\rangle)=(g_0;m_1^\prime n_1^\prime,\dots,m_l^\prime n_l^\prime )$ such that for each $i$, $m'_{i}n'_{i}\neq1$ and $m'_{i}n'_{i}\mid n$.
		\item If $g_0=0$ in condition $(iii)$, then there exist a sub-multiset $A=\{n_{11},\dots,n_{l1}\}$ of the multiset $B = \{m_1^\prime n_1^\prime,\dots,m_l^\prime n_l^\prime\}$ such that $\lcm(\widehat{A})=\lcm(\{n_{11},\dots,\widehat{n_{i1}},\dots,n_{l1}\})=n$ and $m\mid\lcm(B \setminus A)$.
		\item 
		\begin{enumerate}[(a)]
			\item Denoting $\lcm(\{m_k'n_k' : m_{k}'\neq1\})=B_1$, if $\displaystyle \sum_{n_i'\neq 1}\frac{n}{gcd(n,n_i'm_i')}c_i\equiv -\delta_2 \pmod n$, where $m_i' \in \{1,m^o_1,\dots,m^o_t\}$ and $n_i' \in  \{1,n_1,\dots,n_r\}$, then $\frac{n}{B_1}|\delta_2$.
%			$\frac{1}{m}\sum_{i=1}^{r}\alpha_i\frac{n}{n_i}c_i\equiv -\delta_2 \pmod n$,which is equivalent to,\\
			\item Denoting $\lcm(\{m_{l}^\prime n_{l}^\prime: n_{l}'\neq1\})=\bar{B_2}$, and $\gcd(\bar{B_2},m)=B_2$, if $\displaystyle \sum_{m_i'\neq 1}\frac{m}{gcd(m,m_i'n_i')}d_i \equiv -\delta_1 \pmod m$, where $m_i' \in \{1,m_1,\dots,m_k\}$ and $n_i' \in  \{1,n_1^o,\dots,n_s^o\}$, then $\frac{m}{B_2}|\delta_1$.
\end{enumerate}
\end{enumerate}
\end{defn}

\begin{exmp}
Let $F,G\in \Homeo^+(S_2)$ with  $D_F=(6,0;((1,6),2),(2,3))$, $D_G=(2,0;((1,2),6))$, respectively. Then $(G,F)$ is an essential pair of order $2\cdot6$, with $D_{\bar{F}}=(6,0;(1,6),(5,6))$ and $D_{\bar{G}}=(2,0;((1,2),2))$, where $\Gamma(\mathcal{O}_{\langle G \rangle}/\langle\bar{F}\rangle)=\Gamma(\mathcal{O}_{\langle F \rangle}/\langle\bar{G}\rangle) = (0;2,6,6)$. It is easy to check that $(G,F)$ is also a weak abelian pair of order $2\cdot6$. 
\end{exmp}

\noindent Given a finite set $S$ of positive integers, we denote the least common multiple of the integers in $S$ by $\lcm(S)$. In order to improve the clarity of exposition, we will divide the proof our main result into four subcases, of which the first two cases (that will form bulk of our proof) assume the following condition on the quotient orbifolds (of the cyclic factor subgroups).  

\begin{defn}
Let $H < \Homeo^+(S_g)$ be a finite cyclic group, and let~$\Gamma(\O_{H}) = (g_0;n_1,\ldots ,n_{\ell})$. We say the action of $H$ on $S_g$ satisfies the \textit{lcm condition} if $$\lcm(\{n_1,\ldots,n_{\ell}\}) = |H|.$$
\end{defn}

\noindent We conclude this section with another lemma that will be used in one of the subcases of our main result. 

\begin{lemma}
\label{lem:exist_of_lcm_condn}
 Let $F,G \in \Homeo^+(S_g)$ be of orders $n \text{ and }m$, respectively. If $FG = GF$ and $S_g/\langle F,G\rangle \approx S_0$, then there exists a $F' \in \langle F,G\rangle$ of order $n$ such that the action of $\langle F' \rangle$ on $S_g$ satisfies the lcm condition.
\end{lemma}
\begin{proof}
Let $H = \langle G \rangle$. Consider the map $\bar{F} \in \text{Aut}(\O_H)$ induced by $F$. Since $\O_H/\langle \bar{F}\rangle=S_g/\langle F,G\rangle$, the action of $\bar{F}$ on $\O_H$ satisfies the lcm condition. Let $D_{\bar{F}}=(n,0;(c_1',n_1'),\dots,(c_s',n_s'))$. Consider a minimal subset $\{n_{11},\dots,n_{1l}\}$ of the multiset $\{n_1',n_2',\dots,n_s'\}$ with the property $\lcm(\{n_{11},\dots,n_{1l}\})=n$. Now, for each $n_{1i}$, there exists $l_i$ such that $G^{l_i}F^\frac{nc_{1i}}{n_{1i}}(x_i)=x_i$,  for some $x_i\in S_g$. It is apparent that $|G^{l_i}F^\frac{nc_{1i}}{n_{1i}}| \geq n_{1i}$. For each $1 \leq i \leq l$, we choose an appropriate power of $G^{l_i}F^\frac{nc_{1i}}{n_{1i}}$ that we denote by $F_i'$, so that $\gcd(|F_i'|,|F_j'|) = 1$, when $i \neq j$, and $\lcm(\{|F_1'|,\ldots,|F_l'|\} = n$. Thus, the assertion follows by choosing $F' = F_1'F_2'\ldots F_l'$.
\end{proof}

\section{Main theorem}	
\label{sec:main}

By a \textit{two-generator finite abelian action of order $mn$} (written as $m \cdot n$), we mean a tuple $(H, (G,F))$, where 
$m \mid n$ and $H < \Homeo^+(S_g)$, and $$ H = \langle G,F \, |\, G^m = F^n = 1, \, [F,G] = 1\rangle.$$ 
\begin{defn}
Two finite abelian actions $(H_1, (G_1, F_1))$ and $(H_2, (G_2, F_2))$  or order $m \cdot n$ are said to be \textit{weakly conjugate} if there exists an isomorphism, $\psi: \pi_1^{orb}(\O_{H_1}) \cong \pi_1^{orb}(\O_{H_2})$ and an isomorphism $\chi : H_1 \to H_2$ such that
\begin{enumerate}[(i)]
\item $\chi((G_1,F_1)) = (G_2,F_2)$,
\item $(\chi \circ \phi_{H_1})(g) = (\phi_{H_2} \circ \psi)(g), \text{ whenever } g \in \pi_1^{orb}(\O_{H_1}) \text{ is of finite-order, and}$
\item the pair $(G_1, F_1)$ is conjugate (component-wise) to the pair $(G_2, F_2)$ in $\Homeo^+(S_g).$
\end{enumerate}
\noindent The notion of weak conjugacy induces an equivalence relation on the two-generator finite abelian subgroups of $\Homeo^+(S_g)$, and we will call the equivalence classes as \textit{weak conjugacy classes}. 
\end{defn}

\noindent We will now define an abstract tuple of integers that encode, as we will see shortly in  Proposition~\ref{abelian condition thm}, the weak conjugacy class of a two-generator finite abelian action. 
\begin{defn} 
\label{defn:ab_data_set}
An \textit{abelian data set} of \textit{degree} $m \cdot n$ and \textit{genus g} is a tuple 
$$(m\cdot n,g_0;[(c_{11},n_{11}),(c_{12},n_{12}),n_1],\dots,[(c_{r1},n_{r1}),(c_{r2},n_{r2}),n_r]),$$ where $m,n \geq 2$, $g_0 \geq 0$, and $g \geq 2$ are integers satisfying the following conditions:
	\begin{enumerate}[(i)]
	\item $m \mid n$,
		\item $\displaystyle \frac{2g-2}{mn}=2g_0-2+\sum_{i=1}^{r}\left(1-\frac{1}{n_i}\right),$
		\item $\lcm (n_1,\dots,n_r) = \lcm(n_1,\dots,\widehat{n}_k,\dots,n_r) = N,$ and $\text{if } g_0 = 0, \text{ then } N=n,$
		\item for each $i$, $n_{i1}|m$, $n_{i2}|n$, and $\lcm(n_{i1},n_{i2})=n_i$,
		\item for each $i,j$, either $(c_{ij},n_{ij})=1$, or $c_{ij}=0$, and $c_{ij}=0$, if, and only if, $n_{ij}=1$,
		%which implies both $(c_{i1},c_{i2})$ can't be 0.
		\item $\displaystyle \sum_{i=1}^{r}\frac{m}{n_{i1}} c_{i1} \equiv 0 \pmod{m}$ and $\displaystyle \sum_{i=1}^{r}\frac{n}{n_{i2}} c_{i2} \equiv 0 \pmod{n}$, and
		\item when $g_0=0$, there exists $\,(\ell_1,\dots,\ell_{r})$, $(k_1,\dots,k_{r}) \in \mathbb{Z}^{r}$ such that
\begin{enumerate}[(a)]
		\item $\displaystyle \sum_{i=1}^{r}\frac{n}{n_{i1}}c_{i1}\ell_i \equiv 0 \pmod{n}$ and $\displaystyle\sum_{i=1}^{r}\frac{m}{n_{i2}}c_{i2}\ell_i \equiv 1 \pmod{m}$, and 
		\item $\displaystyle \sum_{i=1}^{r}\frac{n}{n_{i1}}c_{i1}k_i \equiv 1 \pmod{n}$ and $\displaystyle \sum_{i=1}^{r}\frac{m}{n_{i2}}c_{i2}k_i \equiv 0 \pmod{m}.$
		\end{enumerate}
	\end{enumerate}
\end{defn}

\begin{prop}
\label{abelian condition thm}
For $m,n,g \geq 2$ and $m \mid n$, abelian data sets of degree $m\cdot n$ and genus $g$ correspond to the weak conjugacy classes of $\mathbb{Z}_m\oplus\mathbb{Z}_n$-actions on $S_g$.
\end{prop}
\begin{proof}
Let $D$ be an abelian data set of degree $m\cdot n$ and genus $g$ as in Definition~\ref{defn:ab_data_set}. By Lemma~\ref{Harvey Condition}, it suffices to show there exists a surjective map $\phi :\pi_1^{orb}(\O_H) \to H$ that preserves the order of torsion elements, where $H = \mathbb{Z}_m \oplus \mathbb{Z}_n$ and $\Gamma(\O_H) = (g_0; n_1, \ldots, n_r)$. Let the presentation of $\Gamma$ and $\mathbb{Z}_m\oplus\mathbb{Z}_n$ be given by 
\begin{gather*}
\langle \alpha_1,\beta_1,\dots,\alpha_{g_0},\beta_{g_0}, \xi_1,\dots,\xi_{\ell} \, |\,  
\xi_1^{n_1}=\dots=\xi_\ell^{n_{\ell}}=\prod_{i=1}^{\ell} \xi_i \prod_{i=1}^{g}[\alpha_i,\beta_i] =1\rangle \text{ and} \\
\mathbb{Z}_m\oplus\mathbb{Z}_n = \langle x,y \, | \, x^m=y^n=[x,y]=1 \rangle.
\end{gather*}
First, we show the result for the case when $g_0 = 0$. We consider the map $$\xi_i\to x^{\frac{m}{n_{i1}}c_{i1}}y^{\frac{n}{n_{i2}}c_{i2}}, \text{ for } 1 \leq i \leq r.$$
Since $|x^{\frac{m}{n_{i1}}c_{i1}}| = n_{i1}$ and $|y^{\frac{n}{n_{i2}}c_{i2}}| = n_{i2}$, condition $(iv)$ implies that $\phi$ is an order-preserving map. Moreover, condition $(vi)$ implies that $\phi$ satisfies the long relation $\prod_{i=1}^{r}\xi_i=1$. In order to show that $\phi$ is surjective, we establish that 
$\phi(\Gamma)$ generates the group $\mathbb{Z}_m\oplus\mathbb{Z}_n$. But condition (vii) ensures that $\{\phi(\xi_i): 1 \leq i \leq r\}$ generates $\mathbb{Z}_m \oplus \mathbb{Z}_n$, and hence it follows that $D$ determines a $\mathbb{Z}_m \oplus \mathbb{Z}_n$-action on $S_g$. When $g_0 > 0$, $\pi_1^{orb}(\O_H)$ also has hyperbolic generators (i.e. the $\alpha_i$ and the $\beta_i$), which can be mapped surjectively to the generators of $\mathbb{Z}_m \oplus \mathbb{Z}_n$.

Conversely, suppose that there is a $\mathbb{Z}_m \oplus \mathbb{Z}_n$-action $D$ on $S_g$ such that $\mathcal{O}_D$ had genus $g_0$. Then by Theorem~\ref{Harvey Condition}, there exists a surjective homomorphism 
$$\phi :\Gamma \to \mathbb{Z}_m\oplus\mathbb{Z}_n : \xi_i\mapsto x^{\frac{m}{n_{i1}}c_{i1}}y^{\frac{n}{n_{i2}}c_{i2}}, \text{ for } 1 \leq i \leq r,$$ that is order-preserving on the torsion elements. This yields an abelian data set of degree $m \cdot n$ and genus $g$ as in Definition~\ref{defn:ab_data_set}, and the result follows.
\end{proof}

\begin{exmp}
The weak conjugacy classes of the abelian actions illustrated in the first two subfigures of Figure~\ref{fig:k4_s4} (in Section~\ref{sec:intro}) are represented by the abelian data sets
\begin{gather*}
(2 \cdot 2,2;[(0,1),(1,2),2],[(1,2),(0,1),2],[(1,2),(1,2),2]) \text{ and }\\
(2 \cdot 2,1;[(0,1),(1,2),2],[(1,2),(0,1),2],[(1,2),(1,2),2]_5),
\end{gather*}
where the suffix $5$ in the second data set denotes the multiplicity of the subtuple $[(1,2),(1,2),2]$. We will discuss such actions in more detail in Section~\ref{sec:appl}.
\end{exmp}

\noindent  To each $F \in \Mod(S_g)$ of order $n$, we may associate a standard representative $\widetilde{F} \in \Homeo^+(S_g)$ of the same order whose conjugacy class we denote by $D_F$. 

\begin{defn}
\label{defn:weak_commute}
Two elements of a group $G$ are said to \textit{weakly commute} if there exists representatives in their respective conjugacy classes that commute. 
\end{defn}

\noindent For a group $G$, if $g,h \in G$ weakly commute, then we denote it by $\lb g,h \rb = 1$. It is clear from Definition~\ref{defn:weak_commute} that if $\lb g,h \rb \neq 1$, then $g$ and $h$ cannot commute in $G$. 

\begin{rem}
\label{rem:wc-wc}
It follows immediately from Definition~\ref{defn:weak_commute} and the Nielsen-Kerckhoff theorem that given $F,G \in \Homeo^+(S_g)$ of finite-order, $\lb F,G \rb = 1$ if, and only if, as mapping classes, they satisfy $\lb F, G \rb = 1$ in $\Mod(S_g)$. 
\end{rem}

\noindent The proof of the main theorem we will also require the following elementary number-theoretic lemma. 

 \begin{lemma}\label{factor lemma}
	Let $\delta\in \mathbb{Z}_n$, and $k_1,\dots, k_r$ are positive integers such that $\lcm(\{k_1,\dots, k_r\})=\beta \mid n$. If $\frac{n}{\beta}|\delta$, then there exists $\delta_1,\dots,\delta_r \in \mathbb{Z}_n$ such that $\frac{n}{k_i}|\delta_i$ and $\sum_{i=1}^r\delta_i \equiv \delta \pmod n$.
\end{lemma}
\begin{proof}
	Since $\lcm(\{k_1,\dots,k_r\})=\beta$ we have $\gcd(\{\frac{n}{k_1},\ldots,\frac{n}{k_r}\})|\frac{n}{\beta}$. So, there exists integers $c_i$ such that $c= \sum_{i=1}^r c_i \frac{n}{k_i}$. For some integer $t$, if $\delta=ct$, where $c = \gcd(\{\frac{n}{k_1},\ldots,\frac{n}{k_r}\})$, then $\delta = \sum_{i=1}^r tc_i \frac{n}{k_i}$. Taking $\delta_i=tc_i\frac{n}{k_i}$, the assertion follows.
\end{proof}

\noindent We will now state the main result in the paper.

\begin{theorem}[Main Theorem]
\label{main theorem}
	Let $F,G \in Mod(S_g)$ be finite-order maps. Then $\lb F, G \rb =1$ and their commuting conjugates form a two-generator abelian group, if, and only if $(G,F)$ is a weakly abelian pair of order $|G|\cdot |F|$.
\end{theorem}

\begin{proof}
Let $|F| = n$ and $|G| =m$, where $m \mid n$, and let $H = \langle F \rangle$. Let $D_F =  (n,g_1,r_1;((c_1,n_1),\alpha_1),\dots,((c_r,n_r),\alpha_r))$ and~$D_G = (m,g_2,r_2;((d_1,m_1),\beta_1),\linebreak \dots,((d_k,m_k),\beta_k))$, respectively. First, we assume that $\lb F,G \rb =1$, and show that $(G,F)$ form a weakly abelian pair of order $m \cdot n$. Without loss of generality, we may assume that $F$ and $G$ commute in $\Mod(S_g)$. Further, by the Nielsen-Kerckhoff theorem, we may assume up to isotopy that $F$ and $G$ commute in $\Homeo^+(S_g)$. Then by Lemma~\ref{commuting quationt}, it follows that $(G,F)$ forms an essential pair of order $m \cdot n$. It remains to show that $(G,F)$ is a weakly abelian pair as in Definition~\ref{defn:weak_ab_pair}.  Condition $(i)$ in this definition is a consequence of Proposition~\ref{ind condition}, while condition (ii) is a direct consequence of condition $(vii)$ of Definition~\ref{defn:ab_data_set}.  To show condition (iii), it suffices to consider the case when
$$D_F=(n,g_1; ((c_{1},n_{1}),m),\dots,((c_{r},n_{r}),m)),$$ as all other cases follow from similar arguments.
First, we note that $G$ induces a $\bar{G} \in \text{Aut}(\O_H)$ which does not fix any cone point of $\O_H$. Let $\Gamma(\O_H/ \langle \bar{G} \rangle) = (g_0;n_1,\dots,n_r, n_{r+1},\dots,n_{r+l})$. Following the notation in the proof of Theorem~\ref{abelian condition thm}, we map $\xi_i \mapsto F^{\frac{n}{n_i}c_i}$, for $1 \leq i \leq r$. The group relation $\prod_{i=1}^{r+l}\xi_i=\prod_{j=1}^{g_0}[\alpha_j,\beta_j]$ of $\pi_1^{orb}(\O_H/ \langle \bar{G} \rangle)$ would now imply that $\prod_{i=1}^{r+l}\phi(\xi_i)=1 $. Thus, either $\sum_{i=1}^{r}\frac{n}{n_i}c_i=0$, or if $\sum_{i=1}^{r}\frac{n}{n_i}c_i \neq 0$, then condition $(iii)$ is necessary.

Conversely, suppose that $(G,F)$ forms a weakly abelian pair of degree $m \cdot n$ as in Definition~\ref{defn:weak_ab_pair}. By Remark~\ref{rem:wc-wc}, it suffices to show that our assumption yields an abelian data set as desired. 

\noindent \textit{Case 1:} Let $\lcm(\{n_1,\dots,n_r\})=n$. We further assume that $m_i'n_i' = B_1$, where $m_i' \neq 1$, for some $i$. We may assume, without loss of generality, that $i =1$. Then we show that the tuple
\begin{small}
	{\begin{gather*} 
		(m*n,g_0;
		\left[\left(d^\prime_1,m^\prime_1\right),\left(\frac{\alpha c^\prime_1+\delta}{\alpha\kappa},\frac{m^\prime_1n^\prime_1}{\kappa}\right),m^\prime_1 n^\prime_1\right], \\ \left[\left(d^\prime_2,m^\prime_2\right),\left(\frac{c^\prime_2}{\kappa_2},\frac{m^\prime_2n^\prime_2}{\kappa_2}\right),m^\prime_2 n^\prime_2\right] ,\ldots, \left[\left(d^\prime_l,m^\prime_l\right),\left(\frac{c^\prime_l}{\kappa_l},\frac{m^\prime_l n^\prime_l}{\kappa_l}\right),m^\prime_ln^\prime_l\right] ), 
		\end{gather*}}
\end{small}
where $\gcd(c^\prime_{j},m^\prime_{j})=\kappa_j$, $\kappa=\gcd(c_{1}+\frac{\delta}{\alpha},m^\prime_{1}n^\prime_{1})$, $\alpha=\frac{n}{m^\prime_{t_1}n^\prime_{t_1}}$, $d^\prime_{t_i}=0,\text{ if }m^\prime_{i}\notin \{m^o_1,m^2_o,\dots,m^o_t\}$, and $c^\prime_{i}=0,\text{ if }n^\prime_{i}\notin \{n_1,n_2,\dots,n_r\} $, forms an abelian data set.  Conditions (i) - (iii) of Definition~\ref{defn:ab_data_set} follow directly from our hypothesis. Moreover, for each $i$, we have $\gcd(d_i^\prime, m_i^\prime)=1\text{ and } \gcd\left(\frac{c^\prime_i}{\kappa_i},\frac{m^\prime_in^\prime_i}{\kappa_i}\right)=1$, and by our choice of $\kappa_i$, we have $\lcm(m_i^\prime,\frac{m^\prime_in^\prime_i}{\kappa_i} )=m^\prime_in^\prime_i$, from which conditions (iv) and (v) follow. Furthermore, our choice of $c_i^\prime$ and $\delta_2$ ensures that
	\begin{equation*} 
	\sum_{i=1}^{l}\frac{n}{m^\prime_in^\prime_i}c_i^\prime + \delta_2\equiv 0 \pmod n \text{ and }
	\sum_{i=1}^{l}\frac{m}{m^\prime_i }d_i^\prime  \equiv 0 \pmod m,
	\end{equation*} which yields condition (vi). It now remains to show (vii), when $g_0=0$. Following the notation used in the proof of Theorem~\ref{abelian condition thm}, we show that the generators $y,x$ (of $\mathbb{Z}_m \oplus \mathbb{Z}_n$) can be expressed as products of elements in the set $\{\phi(\xi_i):1\leq i \leq l\}$. Consider the set $S=\{\phi(\xi_i)^{m_i}:1\leq i \leq l\}$. Then  by our choice of the map $\phi$, each element of $S$ equals some power of $x$, and $|\phi(\xi_i)^{m_i}| = n_i$. Since $\lcm{(n_1,\dots,n_l)}=n$, we have $\langle S \rangle = \langle x \rangle$. Now consider the set $T= \{\phi(\xi_r): \phi(\xi_r) = y^ax^b, \, a\neq 0\}$. Since $(G,F)$ is an essential pair, $yx^t$ is a product of elements in $T$, and the assertion follows.
	
	Now suppose that $\lcm(\{m_k'n_k' : m_{k}'\neq1\})=B_1$, where no $m_k'n_k'$ equals $B_1$. Without loss of generality, we may assume that $\lcm(\{m_k'n_k' : m_{k}'\neq1 \text{ and } 1 \leq k \leq p\})=B_1$. Then by Lemma~\ref{factor lemma}, there exists $\delta_i^\prime$, for $1 \leq i \leq p$, such that $ \sum_{i=1}^p \delta_i' \equiv \delta_2 \pmod{n}$. For each $\delta_i'$, we choose $\alpha_i= \frac{n}{m^\prime_in^\prime_i}$ and consider the tuple 
 \begin{tiny} 
 \begin{gather*}
	(m*n,g_0;
\left[\left(d^\prime_1,m^\prime_1\right),\left(\frac{\alpha_1 c^\prime_1+\delta_1'}{\alpha_1\xi_1^\prime},\frac{m^\prime_1n^\prime_1}{\xi_1^\prime}\right),m^\prime_1n^\prime_1\right], \ldots, \left[\left(d^\prime_p,m^\prime_p\right),\left(\frac{\alpha_p c^\prime_p+\delta_p'}{\alpha_p\xi_p^\prime},\frac{m^\prime_pn^\prime_p}{\xi_p^\prime}\right),m^\prime_pn^\prime_p\right], \\
\left[\left(d^\prime_{p+1},m^\prime_{p+1}\right),\left(\frac{c^\prime_{p+1}}{\xi_{p+1}},\frac{m^\prime_{p+1}n^\prime_{p+1}}{\xi_{p+1}}\right),m^\prime_{p+1}n^\prime_{p+1}\right], \ldots, \left[\left(d^\prime_l,m^\prime_l\right),\left(\frac{c^\prime_l}{\xi_l},\frac{m^\prime_ln^\prime_l}{\xi_l}\right),m^\prime_l n^\prime_l\right] ),\tag{*}
	\end{gather*}
	\end{tiny}
\noindent where $\xi_j^\prime=\gcd(\{c_{j}+\frac{\delta_j'}{\alpha_j},m^\prime_{j}n^\prime_{j}: 1 \leq j \leq p\})$ and $\gcd(c^\prime_{i},m^\prime_{i})=\xi_i$, for $p+1 \leq i \leq l$. As before, this tuple will satisfy all the conditions of an abelian data set.

\textit{Case 2:} Let $\lcm(\{m_1,\dots,m_k\})=m$ and $\lcm(\{n_1,\dots,n_r\})<n$. By an argument analogous to Case 1, we obtain a representation $\phi: \Gamma \to \mathbb{Z}_m \oplus \mathbb{Z}_n$ such that the generators $y,x$ (of $\mathbb{Z}_m \oplus \mathbb{Z}_n$) can be expressed as products of elements in the set $\{\phi(\xi_i):1\leq i \leq l\}$. Consider the set $S=\{\phi(\xi_i)^{n_i}:1\leq i \leq l\}$. Then by our choice of $\phi$ and Proposition~\ref{cone condition}, it follows that each element of $S$ equals some power of $y$ and $|\phi(\xi_i)^{n_i}| = m_i$. Since $\lcm{(m_1,\dots,m_l)}=m$, we have $\langle S \rangle = \langle y \rangle$. Now consider the set $T= \{\phi(\xi_r): \phi(\xi_r) = y^ax^b, \, b\neq 0\}$. As $(G,F)$ forms an essential pair, $xy^t$ is a product of elements in $T$, and the assertion follows.

 \textit{Case 3:} Let $\lcm(\{m_1,\dots,m_k\}) < m$,  $\lcm(\{n_1,\dots,n_r\}) <  n$, and $g_0 >  1$. Then the abelian data set and the representation $\phi$ from Case 1 also works for this case.
 
 \textit{Case 4:} Let $\lcm(\{m_1,\dots,m_k\}) < m$,  $\lcm(\{n_1,\dots,n_r\}) <  n$, and $g_0 =  0$. Then by Lemma~\ref{lem:exist_of_lcm_condn}, it follows that there exists an $F' \in \langle F, G\rangle$ such that $|F'| = n$ and $D_{F'} = (g_0'; (c_1,n_1), \ldots, (c_r,n_r))$ satisfies $\lcm(\{n_1,\ldots,n_r\}) = n$. Since $(G,F)$ is a weakly abelian pair, so is $(G,F')$, and hence this case reduces to Case 1.
\end{proof}

\section{Applications}
\label{sec:appl}
In this section, we derive several applications of the theory developed in the earlier section.
\subsection{Weak commutativity of involutions} It is well known that the conjugacy class of an involution $F \in \m$ is represented by $D_F = (2,g_0; ((1,2),k)), \text{ where } k = 2(g-2g_0+1),$ if $F$ is a non-free action on $S_g$, ans $D_F = (2,(g+1)/2,1;),$ otherwise. In this subsection, we will derive conditions under which two involutions in $\m$ will weakly commute. 
\begin{cor}
\label{cor:comm_of_invo}
Let $F,G \in \Mod(S_g)$ be involutions such that $$D_F =  (2,g_0',r';((1,2),2k')) \text{ and }D_G= (2,g_0'',r'';((1,2),2k'')),$$ respectively. Then $\lb F, G \rb = 1$ if, and only if, the following conditions hold. 
\begin{enumerate}[(a)]
\item There exists $\bar{G} \in \Homeo^+(S_{g_0'})$ with $D_{\bar{G}} = (2,g_0,r_1;((1,2),2s''))$ such that $g+k''+1\geq2s'' \geq k''.$
\item There exists $\bar{F} \in \Homeo^+(S_{g_0''})$ with $D_{\bar{F}} = (2,g_0,r_2;((1,2),2s'))$ such that $g+k'+1\geq2s' \geq k'.$ 
\end{enumerate}
\end{cor}
\begin{proof}
It suffices to show that conditions (a) - (b) mentioned above hold true if, and only if, $(G,F)$ is a weakly abelian pair. If $(G,F)$ is a weakly abelian pair, then it is apparent that (a) - (b) hold. Conversely, it is easy to see that conditions (a) - (b) imply that $(G,F)$ is an essential pair. It remains to show that conditions (i) - (iii) of Definition~\ref{defn:weak_ab_pair} hold true. A simple application of the Riemann-Hurwitz equation to the four data sets that appear in the statement above leads to a system of (four) linear equations, which can be simplified to yield the condition: 
$$2s'-k' = 2s'' - k'',$$ from which (i)-(ii) follow. When $g$ is odd,  $4 \mid \sum_{i=1}^{r}\alpha_i\frac{n}{n_i}c_i$, and so each $\delta_i$ appearing in (iii) is $0$. If $g$ is even, then as no involution generates a  free action, we have $B_i = 2$. Thus, condition (iii) is satisfied, and the assertion follows.
\end{proof}

\noindent Let the conjugacy classes $D_F =  (2,g_0',r';((1,2),2k')) \text{ and }D_G= (2,g_0'',r''; \linebreak((1,2),2k'')),$ be represented by involutions $F \text{ and }G$, which commute. Then, by Corollary~\ref{cor:comm_of_invo}, we have $D_{FG} =(2,g_0,r''';((1,2),2k))$, where $k = 2s'-k' = 2s'' - k''.$ Using this idea, one can obtain a geometric realization of a Klein 4-subgroup $K_4$ of $\Mod(S_g)$ by obtaining an isometric embedding of $\iota: S_g \hookrightarrow \mathbb{R}^3$ that is symmetric about origin such that $\iota(S_g)$ intersects, the $x$-axis at $2k'$ points, the $y$-axis at $2k''$ points, and the $z$-axis at $2k$ points. It is now apparent that under this embedding the non-trivial elements of $K_4$ are realized as $\pi$-rotations about the three coordinates axes. This property is illustrated in the following example.
\begin{exmp}
	Consider $F,G \in \Mod(S_7 )$ whose conjugacy classes given by $D_F=(2,4,1;)$, $D_G=(2,3;((1,2),4))$, respectively. By the preceding discussion, there exist three possible choices for the conjugacy class of $FG$, namely:
	\begin{enumerate}[(a)]
		\item $D_{FG}=(2,4,1;)$
		\item $D_{FG}=(2,2;((1,2),8))$
		\item $D_{FG}=(2,0;((1,2),16))$
	\end{enumerate}
The realization of the group $\{1,F,G,FG\}$ in each case is given in Figure~\ref{fig:Klein4_ModS7} below. 
\end{exmp}
\begin{figure}[h]
\centering
\begin{subfigure}[b]{0.4\textwidth}
	\labellist
	\tiny
	\pinlabel $\pi$ at 23 70
	\pinlabel $\pi$ at 235 75
	\pinlabel $\pi$ at 134 217
	\pinlabel $G$ at -2 45
	\pinlabel $F$ at 250 48
	\pinlabel $FG$ at 127 245
	\endlabellist
\centering
	\includegraphics[width=.8\textwidth]{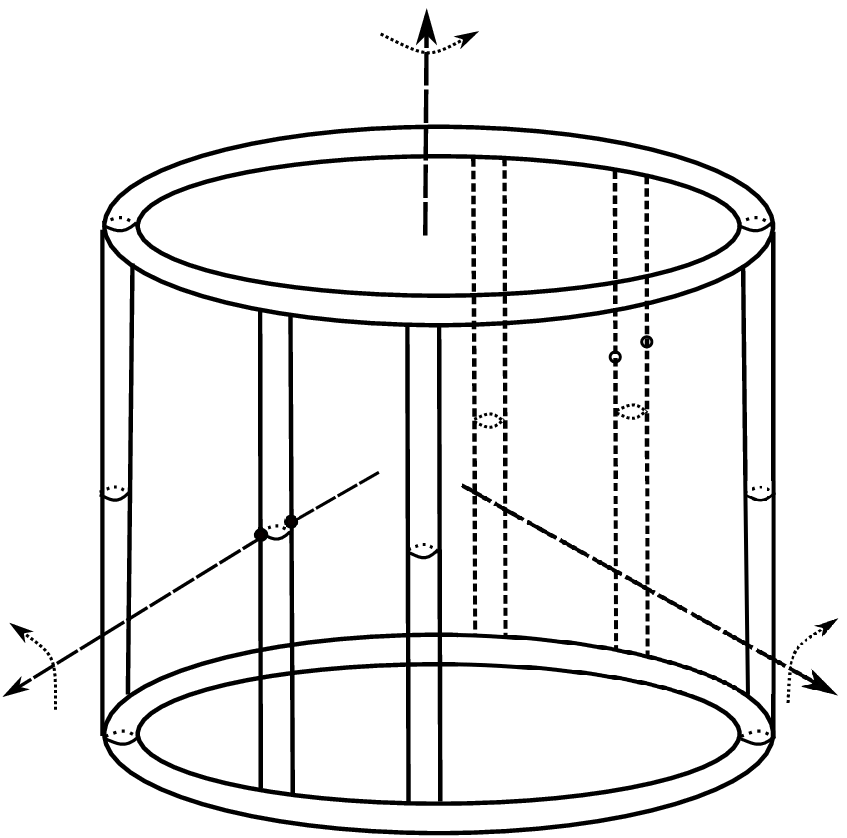}
	\subcaption*{Case (a)}
	\end{subfigure}
	\begin{subfigure}[b]{0.4\textwidth}
	\labellist
	\tiny
	\pinlabel $\pi$ at 20 95
	\pinlabel $\pi$ at 210 95
	\pinlabel $\pi$ at 100 360
	\pinlabel $F$ at -2 55
	\pinlabel $G$ at 90 388
	\pinlabel $FG$ at 245 92
	\endlabellist
	\centering
	\includegraphics[width=.8\textwidth]{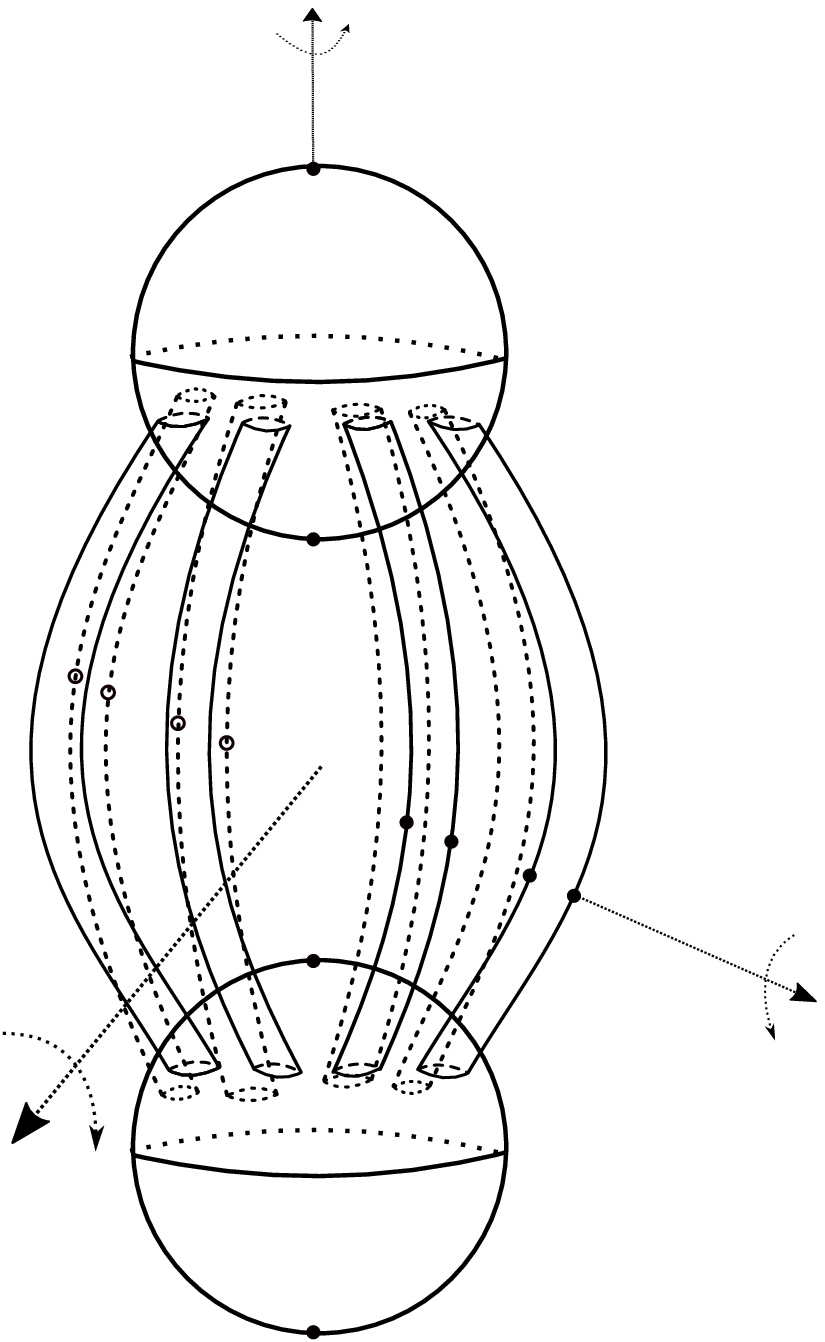}
			\subcaption*{Case (b)}
	\end{subfigure}
	\begin{subfigure}[b]{0.4\textwidth}
	\labellist
	\tiny
	\pinlabel $\pi$ at 405 67
	\pinlabel $\pi$ at 149 13
	\pinlabel $\pi$ at 210 140
	\pinlabel $G$ at 105 02
	\pinlabel $FG$ at 455 76
	\pinlabel $F$ at 195 182
	\endlabellist
\centering
	\includegraphics[width=\textwidth]{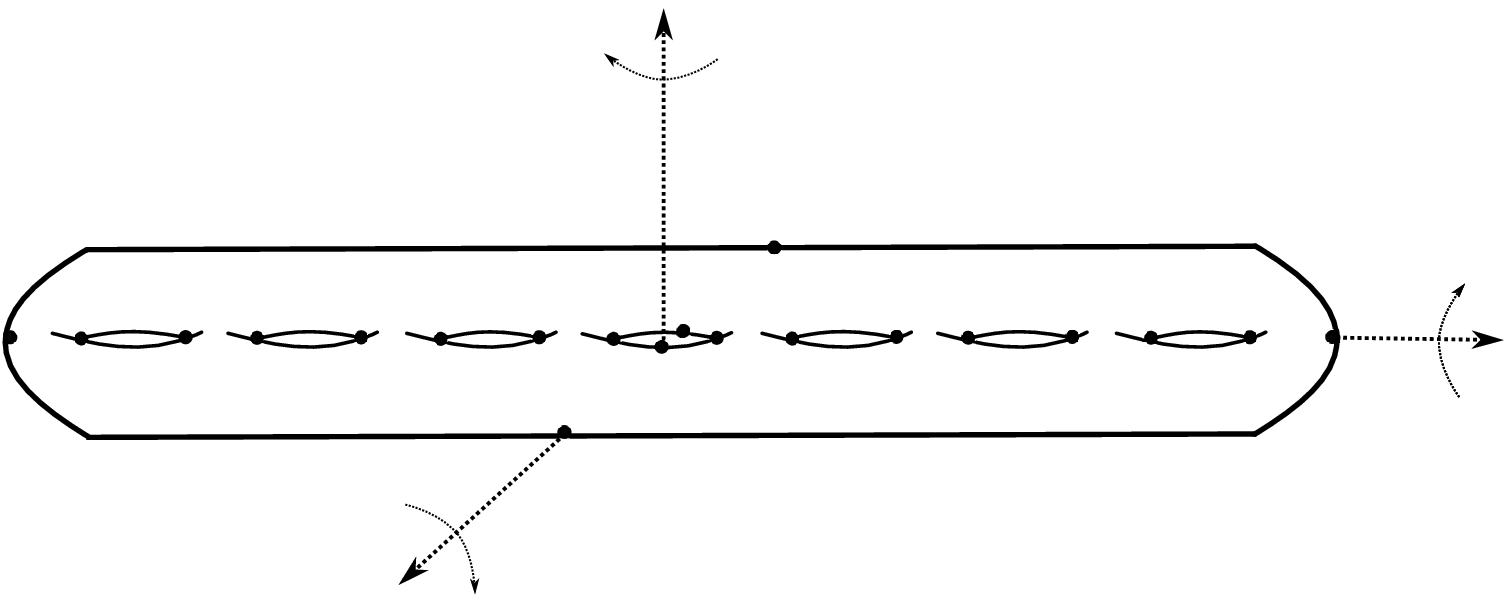}
		\subcaption*{Case (c)}
	\end{subfigure}
	
	\caption{Realizations of three distinct Klein 4-subgroups of $\Mod(S_7)$.}
	\label{fig:Klein4_ModS7}
\end{figure}
\noindent In fact, all Klein $4$-subgroups of $\Mod(S_g)$ can be realized in an analogous manner. 
 
 \subsection{Finite abelain groups with irreducible finite-order mapping classes} We say a $\mathbb{Z}_n$-action is \textit{irreducible} if it is irreducible as a mapping class. By a result of Gilman~\cite{G1}, this is equivalent to requiring that the corresponding orbifold of the action is a sphere with $3$ cone points. Following the nomenclature in~\cite{BPR} and~\cite{PKS}, a $\mathbb{Z}_n$-action on $S_g$ is said to be \textit{rotational} if it can be realized as a rotation about an axis under a suitable isometric embedding of $S_g \hookrightarrow \mathbb{R}^3$. A non-rotational action is said of be of \textit{Type 1} if its quotient orbifold has signature $(g_0;n_1,n_2,n)$, otherwise, it is called a \textit{Type 2} action. The following corollary characterizes the weak commutativity of Type 2 actions with finite-order maps.

\begin{cor}
\label{cor:irr_type2}
There exists no finite non-cyclic abelian subgroup of $\Mod(S_g)$ that contains an irreducible Type 2 action. 
\end{cor}

 \begin{proof}
By Remark~\ref{rem:wc-wc}, it suffices to show that an irreducible Type 2 $\mathbb{Z}_n$-action $F$ does not commute with any other $Z_{m}$-action. Since $F$ is a Type 2 action, we have $\Gamma(\O_{\langle F \rangle}) = (0;n_1,n_2,n_3)$, where $n_i\neq n_j$ and $n_i<n$, for $1\leq i \neq j \leq 3$. In view of Theorem~\ref{main theorem}, if some $G \in \Mod(S_g)$ such that $\lb F, G \rb = 1$, then there exists $\bar{G}:\O_{\langle F \rangle} \to \O_{\langle F \rangle}$ which satisfies the IMP with respect to $(G,F)$. This would imply that $\bar{G}$ fixes all three cone points of $\O_{\langle F \rangle}$. This is impossible, as any homeomorphism on a sphere can fix at most two points, and the assertion follows.
 \end{proof}
 
\noindent We now give a similar characterization for the weak commutativity of Type 1 actions .
 
\begin{cor}
Suppose that there exists a finite non-cyclic abelian subgroup $A$ of $\Mod(S_g)$ that contains an irreducible Type 1 action. Then $A \cong \mathbb{Z}_2 \oplus \mathbb{Z}_{2g+2}.$ 
 \end{cor}

 \begin{proof}
Let $F$ be an irreducible Type 1 action with $\Gamma(\O_{\langle F \rangle}) = (0;n_1,n_2,n_3)$. Since $F$ is of Type 1, there exists atleast one $n_i$ (say $n_1$) such that $n_1=n$, and the following cases arise.
 
 \textit{Case 1: $n_2\neq n_3$ and $n_2,n_3 < n$}.  By an argument analogous to the one used in the proof of Corollary~\ref{cor:irr_type2}, it follows that $F$ does not commute with any other finite-order element of $\m$. 
 
 \textit{Case 2 : $n_i=n$, for $1\leq i\leq 3$.}  Then by the Riemann-Hurwitz equation, we have that $n=2g+1$. By applying a result of Maclachlan~\cite{M2} that bounds the order of a finite abelian subgroup of $\Mod(S_g)$ by $4g+4$, it follows that only an involution can commute with $F$. When such an involution $G$ does commute with $F$, it follows immediately that $\langle F , G \rangle \cong \mathbb{Z}_{4g+2}$.

\textit{Case 3: $n_1 = n_2 =n \neq n_3$.} Once again, by similar arguments as above, we can conclude that $F$ cannot commute with any other finite-order $G \in \m$ with $|G| \geq 3$. When $F$ commutes with an involution $G$, the induced map $\bar{G} \in \text{Aut}(\O_{\langle F \rangle})$ fixes the cone point of order $n_3$ in $\O_{\langle F \rangle}$ and permutes the remaining $2$ cone points. Consequently, we have $\langle F , G\rangle \cong \mathbb{Z}_2\oplus \mathbb{Z}_n$.  By the Riemann-Hurwitz equation, it follows that $n\geq 2g+1$, and hence  $n=2g+2$, as $2n \leq 4g+4$.
 \end{proof}

 \subsection{Weak commutativity with free cyclic actions}
Any non-trivial finite $m$-sheeted cover of $S_g$, for $g \geq 2$, has the form $p : S_{m(g-1)+1} \to S_g$, where $p$ is a covering map. Given such a cover $p$, let $\text{LMod}_p(S_g)$ be the subgroup of $\m$ of mapping classes that lift under the cover. It is natural question to ask whether a given $F \in \m$ of finite-order will have a conjugate $F'$ such that $F' \in \text{LMod}_p(S_g)$. In this subsection, we answer this question for certain types of finite-order maps. We begin by determining when certain types of free cyclic actions weakly commute with other cyclic actions.
 \begin{cor}
 \label{cor:free_comm}
 	Let $F,G \in \Mod(S_g)$ with $D_F = (n,g_1,r;)$ and $D_G = (m,g_0,r';((d_1,m_1),\beta_1),\dots,((d_k,m_k),\beta_k))$, respectively. Suppose that $F$ induces a free action on $\O_{\langle G \rangle}$. Then $\lb F, G \rb =1$ if, and only if: 
 	\begin{enumerate}[(i)]
 		\item $\beta_j=0 \pmod{n}$, for $1 \leq j \leq k$, 
 		\item $n|(g_0-1)$, and
 		\item $\sum_{i=1}^{k}\frac{\beta_i}{n}\frac{m}{m_i}c_i\equiv 0 \pmod m$.
 	\end{enumerate}
 \end{cor}
 \begin{proof}
We show that conditions (i) - (iii)  are sufficient, as it follows directly from Theorem~\ref{main theorem} that they are necessary. By conditions (i) - (ii) of our hypothesis, it follows that there exists a free action on $S_{g_0}$, which induces an $\bar{F} \in \text{Aut}(\O_{\langle G \rangle})$. The Riemann-Hurwitz equation and Lemma~\ref{Harvey Condition} imply that there exists a $\bar{G} \in \text{Aut}(\O_{\langle F \rangle})$ with $$D_{\bar{G}} =(m,\frac{(g_0-1)}{n}+1,r'';((d_1,m_1),\frac{\beta_1}{n}),\dots,((d_k,m_k),\frac{\beta_k}{n})).$$ Hence, it follows that $(G,F)$ forms an essential pair, and the fact that they form an abelian pair follows directly from our hypothesis. 
 \end{proof}
 
 \noindent In the following result, we show that a finite-order mapping class whose corresponding orbifold has genus $>0$ has a conjugate that is liftable under a finite-sheeted cover of $S_g$.
 
 \begin{cor}
 Consider an $F \in \Mod(S_g)$ of finite-order such that $\O_{\langle F \rangle} \not \approx S_0$. Let $p : S_{m(g-1)+1} \to S_{g}$ be an $m$-sheeted cover whose deck transformation group is given by $\langle G \rangle \cong \mathbb{Z}_m$. Then there exists a conjugate $F'$ of $F$ such that $F' \in \text{LMod}_p(S_g)$. 
 \end{cor}
 \begin{proof}
Let $D_F = (n,g_0,r;(c_1,n_1),\dots,(c_r,n_r))$. Then by Corollary~\ref{cor:free_comm}, we have that $\bar{F} \in \Mod(S_{m(g-1)+1})$ with $D_{\bar{F}} = (n,m(g_0-1)+1,\bar{r};((c_1,n_1),m), \linebreak\dots,((c_r,n_r),m))$ such that $\lb G, \bar{F} \rb=1$.  Without loss of generality, we may assume that $G$ and $\bar{F}$ commute in $\Homeo^+(S_g)$. By the IMP, it now follows that $\bar{F}$ induces $F' \in \Mod(S_g)$ that is conjugate to $F$.
 \end{proof}
 
 \noindent In the following corollary, we provide conditions under which certain finite-order mapping class whose corresponding orbifolds are spheres have conjugates that lift under a finite cover of $S_g$.
 
 \begin{cor}
 \label{cor:lift_sphere}
 Let $F \in \Mod(S_g)$ with $D_F = (n,0;(c_1,n_1),\dots,(c_r,n_r))$. Let $p : S_{m(g-1)+1} \to S_{g}$ be an $m$-sheeted cover whose deck transformation group is given by $\langle G \rangle \cong \mathbb{Z}_m$. Then there exists a conjugate $F''$ of $F$ such that $F'' \in \text{LMod}_p(S_g)$, if the following conditions hold. 
 \begin{enumerate}[(i)]
 \item $m \mid n_1$ and $m \mid n_2$.
 \item For $k=1,2$, there exists residue classes $c_k'$ modulo $(n_k/m)$ such that $\gcd(c_k',n_k/m)=1$  and the tuple $$D = (n,0;(c_1',n_1/m),(c_2',n_2/m),((c_3,n_3),m),\dots,((c_r,n_r),m))$$ forms a data set.
 \end{enumerate}
 \end{cor}
 \begin{proof}
Consider an $F' \in \Mod(S_{m(g-1)+1})$ with $D_{F'} = D$.  It is straightforward to check that $(G,F')$ forms a weakly abelian pair. Thus, by Theorem~\ref{main theorem}, we have that $\lb F', G \rb = 1$. So, $F'$ induces $F'' \in \m$ that is conjugate to $F$.
 \end{proof}

\subsection{Primitivity of finite-order mapping classes}
	Let $G$ be group, we say an $x\in G$ has \textit{root of degree n} if there exists $y\in G$ such that $y^n=x$. If a $g \in G$ has no root of any degree greater than one, then $g$ is said to be \textit{primitive} in $G$. It is known~\cite{AW} that the order a finite cyclic subgroup of $\m$ is bounded above by $4g+2$. This would imply that no finite-order mapping class with order $>2g+1$ can have a nontrivial root. The following proposition gives conditions under which an arbitrary finite-order mapping class can have a root.

\begin{prop}\label{primitivity}
	Let $F\in \m$ with $D_F = (n, g_0,r;(c_1,n_1),\dots,(c_r,n_r))$, and let $H = \langle F \rangle$. Then $F$ has a root of degree $m$ if, and only if the following conditions hold. 
	\begin{enumerate}[(i)]
		\item There exists a $G \in \Homeo^+(S_g)$ with $D_G = (m,g',r'; (d_1,m_1),\dots, \linebreak(d_k,m_k))$, which induces a $\bar{G} \in \text{Aut}(\O_H)$. 
		\item $\Gamma(\O_H/\langle \bar{G} \rangle) = (g';n_1',\dots,n_l')$, where for each $i$, $n_i'$ belongs to the following union of multisets
		$$\{n_1,\dots,n_r\}\cup\{m_i|\gcd(m_i,n)=1\}\cup\{n_jm_i|\gcd(m_i,n)=1\}\cup\{nm_j\}.$$
		\item There exist a $F' \in \m$ with $D_{F'} = (mn,g',r'';(c_1',n_1'),\dots,(c_l',n_l'))$ such that for each $i$,
		$$c_i'\equiv \begin{cases}
		c_j, & \text{ if   } n_i'=n_j, \text{ and }\\
		c_j\pmod{n_j}, & \text{ if   } n_i'=n_jm_i.
		\end{cases}$$
	\end{enumerate}
\end{prop}
\begin{proof}
First, we note that the conjugacy of $(F')^m$ is represented by $D_F$. Thus, we have that $(F')^m$ and $F$ are conjugate. So, we can find a conjugate of $F'$, say $\widetilde{F}$, such that $\widetilde{F}^m=F.$ Hence, $F$ has a root of order $m$.
	
Conversely, suppose that $F$ has a root $F'$ of order $m$. Then we show that conditions (i) - (iii) hold. 
Since $F'$ commutes with $(F')^m$, the map $\bar{F}([x])=[F'(x)]$ defines an automorphism of $\O_H$. Furthermore, we have that $$\Gamma(\mathcal{O}_H/\langle \bar{F}\rangle) = \Gamma(S_g/\langle {F'}\rangle) =(g';t_1,t_2,\dots,t_l).$$ Note that, $$t_i\in \{n_1,\dots,n_r\}\cup\{ m_1,\dots,m_k\}\cup\{n_im_j|1\leq i\leq r,1\leq j\leq k\}.$$ So, it remains to prove if $t_i=m_j$, then $\gcd(m_j,n)=1$, and if $t_i=m_jn_k$ then either $n_k=n$ or $\gcd(m_j,n)=1$. However, this follows directly from the structures of $D_{F'}$ and $D_F$.
\end{proof}
\noindent A consequence of this theorem is the following corollary, which pertains to the roots of a mapping class of order $g-1$. 
\begin{cor}
\label{cor:prim_of_free}
Let $F \in \Mod(S_g)$ be represented by the generator of a free cyclic action on $S_g$, and let $F'$ be a nontrivial root of $F$ (if it exists). Then:
	\begin{enumerate} [(i)]
		\item $\O_{\langle F' \rangle} \not\approx S_0$, and
		\item when $|F|= g-1$, $F'$ exists if, and only if, $2 \nmid (g-1)$. Moreover, $F'$ is a root of degree $2$. 		
	\end{enumerate}
\end{cor}
\begin{proof}
	\begin{enumerate}[(i)]
		\item Suppose that $\O_{\langle F' \rangle} \approx S_0$. Then, as discussed in the proof of Proposition~\ref{primitivity}, all its powers of prime order have a fixed point. 
		\item Let $F$ define a free action on $S_g$, and $H=\langle F \rangle$. Then $\O_H \approx S_2$, and by condition (i) of Theorem~\ref{primitivity}, $F'$ induces an $\bar{F} \in \text{Aut}(\O_H)$ of order $n$. In view of (i), it is clear that $n \leq 2$, and further, by condition (ii) of Theorem~\ref{primitivity}, this is only possible when $2\nmid (g-1)$. If $2\nmid (g-1)$, then it easy check that $F'$ with $D_{F'} = (2g-2,1;(1,2),(1,2))$ is a root of $F$ of degree $2$. 
		\end{enumerate}
\end{proof}

\noindent By arguments similar to those in Corollary~\ref{cor:prim_of_free}, we can show that:
\begin{cor}
If $6\mid g$, then an $F \in \Mod(S_g)$ with $D_{F} = (g,1;(c,g),(g-c,g))$ is primitive.
\end{cor}

\subsection{Weak commutativity of finite-order maps with the roots of Dehn twists} Let $c$ be a simple closed curve in $S_g$, for $g \geq 2$, and let $t_c \in \Mod(S_g)$ denote the left-handed Dehn twist about $c$. A \textit{root of $t_c$ of degree $n$} is an $F \in \m$ such that $F^n = t_c$. Consider an $F \in \m$ that is either an order-$n$ mapping class that preserves $c$, or a root of $t_c$ of degree $n$. Then up to isotopy, we can assume that $F(c) = c$, and  that $F$ preserves a closed annular neighborhood $N$ of $c$. Let $\widehat{S_g(c)}$ denote the surface obtained by capping off the components of $\overline{S_g \setminus N}$. Then by the theory developed in~\cite{MK1,KR1,PKS}, it follows that $F$ induces an order-$n$ map $\widehat{F}_c \in \Homeo^+(\widehat{S_g(c)})$ by coning. The following remark describes the construction of a root of a $t_c$, when $c$ is nonseparating.

\begin{rem}
\label{rem:ns_root}
When $c$ is nonseparating, it is well known~\cite{MK1} that (up to conjugacy) a root $F$ of $t_c$ of degree $n$ determines a $\mathbb{Z}_n$-action $\widehat{F}_c$ on $S_{g-1}$, which has two (distinguished) fixed points on $\widehat{S_g(c)}$, where it induces rotation angles add up to  $2 \pi/n \pmod{2\pi}$. (We will call such an action a \textit{nonseparating root-realizing $\mathbb{Z}_n$-action}.) Conversely, consider a $\mathbb{Z}_n$-action on $S_{g-1}$, which has two (distinguished) fixed points, where it induces rotation angles that add up to $2 \pi/n \pmod{2\pi}$. Then we can remove invariant disks around the fixed points and attach a $1$-handle $N$ with an $1/n^{th}$ twist connecting the resulting boundary components to obtain a root of Dehn twist about the nonseparating curve in $N$. 
\end{rem} 

\noindent Moreover, it was shown in~\cite{MK1,KR1} that no root of $t_c$ can switch the two sides of $c$. 

\begin{rem}
\label{rem:self-1-comp}
Suppose that a $\mathbb{Z}_m$-action $G \in \m$ preserves a curve $c$. Then $G$ induces an order-$m$ map $\widehat{G}_c$ on $\widehat{S_g(c)}$. In particular $\widehat{S_g(c)} \approx S_{g-1}$, if $c$ is nonseparating, and $\widehat{S_g(c)} \approx S_{g_1} \sqcup S_{g_2}$ (in symbols $S_g = S_{g_1} \#_c S_{g_2}$), where $g_1+g_2 = g$, when $c$ is separating. Let $N$ be a closed annular neighborhood of $c$ such that $G(N)= N$. Then the two distinguished points $P,Q$ that lie at the center of the capping disks (of the two boundary components of the surface $\overline{S_g \setminus N}$) are either fixed under the action of $\widehat{G}_c$, or form an orbit of size $2$. Conversely, given a $\mathbb{Z}_m$-action $\widehat{G}_c$ on a surface ($\approx \widehat{S_g(c)}$) with two distinguished points $P,Q$, which are either fixed with locally induced rotation angles (around $P$ and $Q$) adding up to $0 \pmod{2 \pi}$, or form a orbit of size $2$, we may reverse the above process to obtain $\mathbb{Z}_m$-action on $S_g$. Note that by~\cite{PKS} $P,Q$ is an orbit of size $2$, only when $|\widehat{G}_c|=2$. 
\end{rem}

\noindent This leads us to the following characterization of weak commutativity of finite-order maps with roots of Dehn twists about nonseparating curves.

\begin{cor}
\label{cor:weak_comm_nonsep_root}
Let $F \in \m$ be a root of $t_c$, where $c$ is nonseparating, and $G \in \m$ be of finite order. Then $\lb F, G \rb = 1$ if, and only if $G(c) = c$ and $\lb \widehat{F}_c, \widehat{G}_{c} \rb= 1$. In particular, if $\widehat{F}_c$ is primitive, then $F$ and $G$ cannot commute in $\m$. 
\end{cor}

\begin{proof}
Suppose that $\lb F,G \rb = 1$. Then up to conjugacy, we assume that $F$ commutes with $G$, and so we have $t_c = Gt_cG^{-1} = t_{G(c)}$. Hence, we may assume up to isotopy that $G(c)=c$, and both $G$ and $F$ preserve the same annular neighborhood $N$ of $c$. Thus, $\widehat{F}_c$ and $\widehat{G}_{c}$, which are induced by $F$ and $G$, respectively, must commute as maps on $S_{g-1}$, and so it follows that $\lb \widehat{F}_c, \widehat{G}_{c'} \rb= 1$.

Conversely, let us assume conditions (i) - (ii) hold true. Then $\widehat{F}_c$ and $\widehat{G}_c$ share the same set of two distinguished points $P$ and $Q$ (as in Remark~\ref{rem:ns_root}) that are either fixed or form an orbit of size $2$, under their actions. By Remarks~\ref{rem:ns_root}-\ref{rem:self-1-comp}, we construct maps $F$ and $G$, which commute in $\Homeo^+(S_g)$. Therefore, as mapping classes they satisfy $\lb F, G \rb = 1$.

Let $H = \langle \widehat{F}_c,  \widehat{G}_c \rangle$. To show the final part of the assertion, we first observe that $\stab_{H}(P) = H$, when $|G| > 2$. Since $H$ is cyclic (by Lemma~\ref{lem:stab_cyc}), it follows that $\widehat{F}_c$ has a root of degree $|G|$. Further, it was shown in~\cite{MK1} that $F$ is always a root of odd degree. So, when $|G| = 2$, it is apparent that $H$ is cyclic. Therefore, if $\widehat{F}_c$ is primitive, then $F$  and $G$ cannot commute in $\m$.
\end{proof}

\noindent Note that the conditions $\gcd(|\widehat{F}_c|,|\widehat{G}_{c}|) =1$ and $|\widehat{F}_c||\widehat{G}_{c}| \leq (4(g-1)+2)$ determine an upper bound for $|G|$.

\begin{rem}
\label{rem:sep-root}
Let $c$ is a separating curve in $S_g$ so that $S_g = S_{g_1} \#_c S_{g_2}$. It is known~\cite{KR2} that (up to conjugacy) a root $F$ of $t_c$ of degree $n$ corresponds to a pair $\widehat{F}_c = (\widehat{F}_{1,c},\widehat{F}_{2,c})$ of finite order maps, where $\widehat{F_{i,c}} \in \Homeo^+(S_{g_i})$ with $|\widehat{F_{i,c}}| = n_i$, for $i = 1,2$, with distinguished fixed points $P_i \in S_{g_i}$ around which the locally induced rotational angles $\theta_i$, which satisfy $$\theta_1 + \theta_2 \equiv 2\pi/n \pmod{2\pi}, \text{ where } n =\text{lcm}(n_1,n_2).$$ 
 Further, if $G$ is a finite-order map with $G(c) = c$ and $|G|>2$, then there is a decomposition of $\widehat{G}_c$ into a pair of actions $(\widehat{G}_{1,c},\widehat{G}_{2,c})$, where $\widehat{G}_{i,c}$ is a $\mathbb{Z}_{m}$-action on $S_{g_i}$, for $i = 1,2$. However, when $|G|=2$, $\widehat{G}_c$ is either a single action on $S_g(c)$ that permutes the components $S_{g_i}$ (in which case $g_1= g_2$), or it decomposes into a pair of actions  $(\widehat{G}_{1,c},\widehat{G}_{2,c})$ as before. 
 \end{rem}
 
\noindent The ideas in Remarks~\ref{rem:self-1-comp} and \ref{rem:sep-root} lead to the following analog of Corollary~\ref{cor:weak_comm_nonsep_root} for the roots of Dehn twists about separating curves. 

\begin{cor}
\label{cor:weak_comm_sep_root}
Let $c$ is a separating curve in $S_g$ so that $S_g = S_{g_1} \#_c S_{g_2}$. Let $F \in \m$ be a root of $t_c$ so that $\widehat{F}_c = (\widehat{F}_{1,c},\widehat{F}_{2,c})$. Then a $G \in \m$ of finite order satisfies $\lb F, G \rb = 1$ if, and only if:
\begin{enumerate}[(i)]
\item $G(c) = c$, and
\item either $\widehat{G}_{c} = (\widehat{G}_{1,c},\widehat{G}_{2,c})$ and $\lb \widehat{F}_{i,c}, \widehat{G}_{i,c} \rb= 1$, for $i=1,2$, or $\widehat{F}_{1,c}$ is conjugate with $\widehat{F}_{2,c}$.
\end{enumerate}
\end{cor}

\section{Hyperbolic structures realizing abelian actions} 
\label{sec:factor_gens}
In~\cite{BPR} and~\cite{PKS}, a procedure to obtain the hyperbolic structures that realize cyclic subgroups of $\Mod(S_g)$ as isometries was described.  In this section, we use this procedure, and theory developed in Sections~\ref{sec:ind_aut}-\ref{sec:main} to give an algorithm for obtaining the hyperbolic structures that realize a given two-generator finite abelian subgroup of $\m$ as an isometry group. Given a finite subgroup $H < \m$, let $\text{Fix}(H)$ denote the subspace of fixed points in the Teich\"{m}uller space $\text{Teich}(S_g)$ under the action of $H$. With this notation in place, we have the following elementary lemma.

\begin{lemma}
\label{lem:fix_abelian}
Let $F,G\in \m$ be commuting finite-order mapping classes. Then $$\text{Fix}(\langle F,G\rangle)=\text{Fix}(\langle F\rangle)\cap\text{Fix}(\langle G\rangle).$$
\end{lemma}
\begin{proof} Suppose that
	$x\in \text{Fix}(\langle F,G\rangle)$. Then $x\in \text{Fix}(\langle F\rangle) \text{ and }x\in \text{Fix}(\langle G\rangle)$, and so $x\in \text{Fix}(\langle F\rangle)\cap\text{Fix}(\langle G\rangle)$
	
	Conversely, given $x\in \text{Fix}(\langle F\rangle)\cap\text{Fix}(\langle G\rangle)$, thus $F(x)=G(x)=x$ so $F^l  G^k (x)=x$, for all $l,k$, which implies that $x\in \text{Fix}(\langle F,G\rangle)$.
\end{proof}

In~\cite{BPR, PKS}, it was shown that:
\begin{theorem}\label{res:1}
For $g \geq 2$, consider a $F \in \m$ with $D_F = (n,g_0; (c_1,n_1),\linebreak (c_2,n_2), (c_3,n))$. Then $F$ can be realized explicitly as the rotation $\theta_F$ of a hyperbolic polygon $\P_F$ with a suitable side-pairing $W(\P_F)$, where $\P_F$ is a hyperbolic  $k(F)$-gon with
$$ \small k(F) := \begin{cases}
2n(1+2g_0), & \text { if } n_1,n_2 \neq 2, \text{ and } \\
n(1+4g_0), & \text{otherwise, }
\end{cases}$$
and for $0 \leq m\leq n-1$, 
$$ \small
W(\P_F) =
\begin{cases}
\displaystyle  
  \prod_{i=1}^{n} Q_ia_{2i-1} a_{2i} \text{ with } a_{2m+1}^{-1}\sim a_{2z}, & \text{if } k(h) = 2n, \text{ and } \\
\displaystyle
 \prod_{i=1}^{n} Q_ia_{i} \text{ with } a_{m+1}^{-1}\sim a_{z}, & \text{otherwise,}
\end{cases}$$
where $\displaystyle z \equiv m+qj \pmod{n}, \,q= (n/n_2)c^{-1},j=n_{2}-c_{2}$, and \\ $\small Q_r = \prod_{s=1}^{g_0} [x_{r,s},y_{r,s}], \, 1 \leq r \leq n.$ Further, when $g_0 = 0$, this structure is unique. 
\end{theorem}

\noindent  Suppose that a $\mathbb{Z}_m$-action on $S_g$ induces a pair of orbits of size $r$, where the induced rotation angles add up to $0 \pmod{2\pi}$. Then we can remove cyclically permuted $\mathbb{Z}_m$-invariant disks around points in the orbits and then identifying the resultant boundary components to obtain a $\mathbb{Z}_m$-action on $S_{g + m-1}$. This construction is called a \textit{self r-compatibility}, and we say that $G$ as above \textit{admits a self r-compatibility}. Conversely given a $\mathbb{Z}_m$-action on $S_{g}$ that admits a self $r$-compatibility, we can reverse the construction described above to recover the $\mathbb{Z}_m$-action on $S_g$. Further, it was shown that a non-rotational Type 2 action can be realized from finitely many pairwise $r$-compatibilities between Type 1 actions. 

Given a weak conjugacy class of an abelian action $(H, (G,F))$ represented by $$(m\cdot n,g_0;[(c_{11},n_{11}),(c_{12},n_{12}),n_1],\dots,[(c_{r1},n_{r1}),(c_{r2},n_{r2}),n_r]),$$ we will now describe an algorithmic procedure for obtaining the conjugacy classes of its generators. Let $H_1 = \langle F \rangle$ and $H_2 = \langle G \rangle$ by applying . 

\begin{enumerate}[\textit{Step} 1.]
\item It follows directly from our theory that the data sets $$D_{\bar{G}} =  (m,g_0;(c_{11},n_{11}),\dots,(c_{r1},n_{r1})) \text{ and } D_{\bar{F}}= (n,g_0;(c_{12},n_{12}),\dots,(c_{r2},n_{r2}))$$ represent the conjugacy classes of the actions $\bar{G}$ and $\bar{F}$ induced on the orbifolds $\O_{H_1}$ and $\O_{H_2}$ by the actions of $H_1$ and $H_2$ on $S_g$, respectively. 

\item We now note that the orbifold signatures $\Gamma(\O_{H_i})$ have the form
\begin{gather*}\small \Gamma(\O_{H_1})=(n,g(D_{\bar{G}});(\underbrace{\frac{n_1}{n_{11}}, \ldots \frac{n_1}{n_{11}}}_{\frac{m}{n_{11}}\text{ times}}, \ldots, \underbrace{\frac{n_r}{n_{r1}}, \ldots \frac{n_r}{n_{r1}}}_{\frac{m}{n_{r1}}\text{ times}})) \text{ and } \\
\small \Gamma(\O_{H_2})=(n,g(D_{\bar{F}});(\underbrace{\frac{n_1}{n_{12}}, \ldots \frac{n_1}{n_{12}}}_{\frac{n}{n_{12}}\text{ times}}, \ldots, \underbrace{\frac{n_r}{n_{r2}}, \ldots \frac{n_r}{n_{r2}}}_{\frac{n}{n_{r2}}\text{ times}})),\end{gather*}
with the understanding that if $n_i/n_{ij} =1$, for some $1 \leq i \leq r$ and $j = 1,2$, then we exclude it from the signatures.

\item We choose conjugacy classes 
$$\small D_F= (n,g_1;( (c_1,\frac{n_1}{n_{11}}),\frac{m}{n_{11}}),\dots,((c_r,\frac{n_r}{n_{r1}}),\frac{m}{n_{r1}})) \text{ and }$$ $$\small D_G= (m,g_2;( (d_1,\frac{n_1}{n_{12}}),\frac{n}{n_{12}}),\dots, ((d_r,\frac{n_r}{n_{r2}}),\frac{m}{n_{r2}})),$$ where $c_i \equiv c_{i2} \pmod{n_i/n_{i1}}$ and $d_i \equiv c_{i1} \pmod{n_i/n_{i2}}.$

\item Finally, using Lemma~\ref{lem:fix_abelian}, Theorem~\ref{res:1}, and the subsequent discussion on the theory developed in~\cite{BPR, PKS}, we can obtain the hyperbolic structures that realize $\langle F, G \rangle$ as group of isometries.
\end{enumerate}

\noindent In Table~\ref{tab:wkc_mods3} at then end of this section, we give a complete classification of weak conjugacy classes of two-generator finite abelian subgroups of $\Mod(S_3)$. Using the algorithm described above, in Figure~\ref{fig:Z2Z4_S3},  we provide a geometric realization of the weak conjugacy classes in S.Nos 10-12. The pairs of integers labeled in each subfigure are the pairs $\P_{[x]}$, which correspond to cone points $[x]$ in the quotient orbifold $\O_{\langle F\rangle}$.

\begin{figure}[H]
	\centering
	\tiny
	\labellist
	\pinlabel $(3,4)$ at 75 -1
	\pinlabel $(3,4)$ at 75 130
	\pinlabel $(1,4)$ at 325 -1
	\pinlabel $(1,4)$ at 325 130
	\pinlabel $(1,2)$ at 150 100
	\pinlabel $(1,2)$ at 254 100
	\pinlabel $(1,2)$ at 150 28
	\pinlabel $(1,2)$ at 254 28
	\pinlabel $\pi$ at 423 78
	\pinlabel $G$ at 440 78
	\endlabellist
	\includegraphics[width=.7\textwidth]{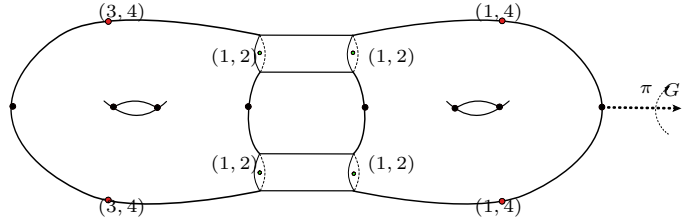}
	\caption{A realization of the action in S.No.10 of Table~\ref{tab:wkc_mods3}, with $D_G=(2,0;((1,2),8)) \text{ and } D_F=(4,0;((1,4),2),((3,4),2))$. $D_F$ can be realized as a 2-compatibility between two actions $F'$ and $F''$, where $D_{F'} = (4,0;(1,2),((3,4),2)) \text{ and } D_{F''} = (4,0;(1,2),((1,4),2))$. Note that $F'$ and $F''$ are realized rotations of the polygons $\P_{F'}$ and $\P_{F''}$ described in Theorem~\ref{res:1}.}
\end{figure}

\begin{figure}[H]
\centering
	\tiny
	\labellist
	\pinlabel $\pi$ at 84 217
	\pinlabel $G$ at 94 220
	\pinlabel $(3,4)$ at 46 173
	\pinlabel $(3,4)$ at 128 180
	\pinlabel $(1,4)$ at 49 34
	\pinlabel $(1,4)$ at 134 67
	\pinlabel $(1,2)$ at 11 170
	\pinlabel $(1,2)$ at 169 170
	\pinlabel $(1,2)$ at 169 45
	\pinlabel $(1,2)$ at 11 45
	\endlabellist
	\includegraphics[width=.4\textwidth]{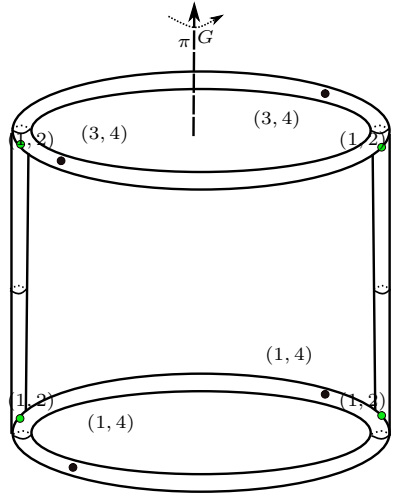}
	\caption{A realization of the action in S.No.11 of Table~\ref{tab:wkc_mods3}, with $D_G=(2,2,1;) \text{ and } D_F=(4,0;((1,4),2),((3,4),2))$. Here, $D_F$ can be realized as a 2-compatibility between two actions $F'$ and $F''$ (realized as before), where $D_{F'} = (4,0;(1,2),((3,4),2)) \text{ and } D_{F''} = (4,0;(1,2),((1,4),2))$. }
\end{figure}

\begin{figure}[H]
\centering
	\tiny
	\labellist
	\pinlabel $\pi$ at 242 107
	\pinlabel $G$ at 255 108
	\pinlabel $(1,4)$ at -10 47
	\pinlabel $(1,4)$ at 504 48
	\pinlabel $(1,2)$ at 42 -3
	\pinlabel $(1,2)$ at 42 95
	\pinlabel $(1,2)$ at 215 89
	\pinlabel $(1,2)$ at 290 89
	\pinlabel $(1,2)$ at 456 -3
	\pinlabel $(1,2)$ at 456 95
	\pinlabel $(3,4)$ at 172 47
	\pinlabel $(3,4)$ at 324 47
	\pinlabel $(1,4)$ at 117 47
	\pinlabel $(1,4)$ at 378 47
	
	\endlabellist
	\includegraphics[width=.8\textwidth]{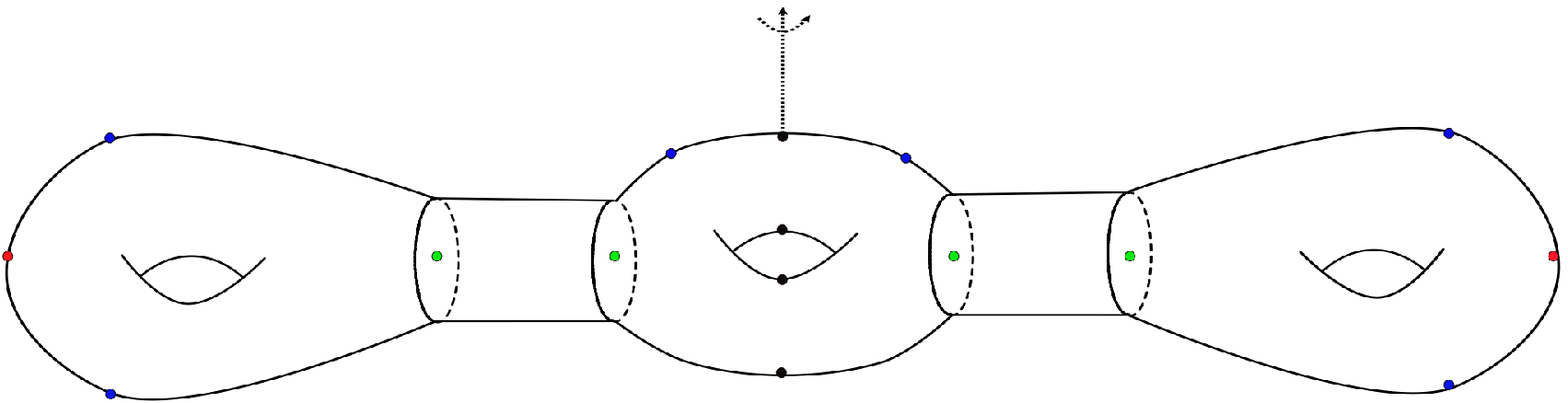}
	\caption{A realization of action in S.No.12 of Table~\ref{tab:wkc_mods3}, with $D_G=(2,1;((1,2),4))$  and $D_F=(4,0;((1,2),3),((1,4),2))$. Here, $D_F$ can be realized by $1$-compatibilities of the two actions $F'$ and $F''$ with $F'''$, where $D_{F'} = (4,0;(1,2),((1,4),2)), D_{F''} = (4,0;(1,2),((1,4),2)), \text{ and } D_{F'''} = (4,0;(1,2),((3,4),2)$. Again $F'$, $F''$ and $F'''$ are irreducible Type 1 actions realized as rotations of polygons described in Theorem~\ref{res:1}.} 
\label{fig:Z2Z4_S3}
\end{figure}

%\section{Classification of weak conjugacy classes in $\Mod(S_3)$}

\begin{table}[H]
\small
\centering
\scalebox{0.70}{
\begin{tabular}{| c | c | c |}
\hline
S.No. &  Abelian Data & Cyclic factors $[D_G;D_F]$\\
\hline 
$1$  & $(2 \cdot 2,1;[(1,2),(0,1),2],[(1,2),(0,1),2])$ & $[(2,2;1;);(2,1;((1,2),4))]$\\
\hline 
$2$ & $(2\cdot2,1;[(0,1),(1,2),2],[(0,1),(1,2),2])$ & $[(2,1;((1,2),4));(2,2;1;)]$\\
				\hline 
$3$ & $(2\cdot2,1;[(1,2),(1,2),2],[(1,2),(1,2),2])$ & $[(2,2;1;);(2,2;1;)]$\\
				\hline 
$4$ & $(2\cdot2,0;[(1,2),(0,1),2]_2,[(1,2),(1,2),2]_4)^*$ & $[(2,1;((1,2),4));(2,2;1;)]$\\
				\hline 
$5$ & $(2\cdot2,0;[(0,1),(1,2),2]_2,[(1,2),(1,2),2]_4)$ & $[(2,1;1;);(2,1;((1,2),4))]$\\
				\hline 
$6$ & $(2\cdot2,0;[(1,2),(0,1),2]_4,[(1,2),(1,2),2]_2)$ & $[(2,0;((1,2),8));(2,2;1;)]$\\
				\hline 
$7$ & $(2\cdot2,0;[(0,1),(1,2),2]_4,[(1,2),(1,2),2]_2)$ & $[(2,2;1;);(2,0;((1,2),8))]$ \\
				\hline 
$8$ & $(2\cdot2,0;[(0,1),(1,2),2]_2,[(1,2),(0,1),2]_4)$ & $[(2,1;((1,2),4));(2,0;((1,2),8))]$ \\
				\hline 
$9$ & $(2\cdot2,0;[(0,1),(1,2),2]_4,[(1,2),(0,1),2]_2)$ & $[(2,0;((1,2),8));(2,1;((1,2),4))]$\\
				\hline 
$10$ & {$(2\cdot4,0;[(1,2),(0,1),2]_2,[(0,1),(1,4),4],[(0,1),(3,4),4])$} & $[(2,0;((1,2),8));(4,0;((1,4),2),((3,4),2))]$\\
				\hline 
$11$ & {$(2\cdot4,0;[(1,2),(1,2),2]_2,[(0,1),(1,4),4],[(0,1),(3,4),4])$} & $[(2,0;1;);(4,0;((1,4),2),((3,4),2))]$ \\
				\hline 
$12$ & $(2\cdot4,0;[(1,2),(0,1),2],[(1,2),(1,2),2],[(1,2),(1,4),4],[(0,1),(1,4),4])$ & $[(2,1;((1,2),4));(4,0;((1,2),3),((1,4),2))]$ \\
				\hline 
$13$ & $(2\cdot4,0;[(1,2),(0,1),2],[(1,2),(1,2),2],[(1,2),(3,4),4],[(0,1),(3,4),4])$ & $[(2,1;((1,2),4));(4,0;((1,2),3),((3,4),2))]$ \\
				\hline 
$14$ & $(2\cdot4,0;[(1,2),(0,1),2]_2,[(1,2),(3,4),4],[(1,2),(1,4),4])$ & $[(2,0;((1,2),8));(4,1;((1,2),2))]$ \\
				\hline 
$15$ & $(2\cdot4,0;[(1,2),(0,1),2],[(1,2),(1,2),2],[(1,2),(1,4),4]_2)$ & $[(2,1;((1,2),4));(4,1;((1,2),2))]$  \\
				\hline 
$16$ &  $(2\cdot4,0;[(1,2),(1,2),2]_2,[(1,2),(3,4),4],[(1,2),(1,4),4])$ & $[(2,2;1;);(4,1;((1,2),2))]$  \\
				\hline 
$17$ & $(2\cdot8,0;[(1,2),(0,1),2],[(1,2),(1,8),8],[(0,1),(7,8),8])$ & $[(2,0;((1,2),8));(8,0;((1,4),1),((7,8),2))]$  \\
				\hline
$18$ & $(2\cdot8,0;[(1,2),(0,1),2],[(1,2),(3,8),8],[(0,1),(5,8),8])$ & $[(2,0;((1,2),8));(8,0;((3,4),1),((5,8),2))]$  \\
				\hline
$19$ & $(2\cdot8,0;[(1,2),(0,1),2],[(1,2),(5,8),8],[(0,1),(3,8),8])$ & $[(2,0;((1,2),8));(8,0;((1,4),1),((3,8),2))]$ \\
				\hline
$20$ & $(2\cdot8,0;[(1,2),(0,1),2],[(1,2),(7,8),8],[(0,1),(1,8),8])$ & $[(2,0;((1,2),8));(8,0;((3,4),1),((1,8),2))]$  \\
				\hline 
$21$ & $(2\cdot8,0;[(1,2),(1,2),2],[(1,2),(1,8),8],[(0,1),(3,8),8])$ & $[(2,0;1;);(8,0;((1,4),1),((3,8),2))]$  \\
				\hline
$22$ & $(2\cdot8,0;[(1,2),(1,2),2],[(1,2),(3,8),8],[(0,1),(1,8),8])$ & $[(2,0;1;);(8,0;((3,4),1),((1,8),2))]$  \\
				\hline
$23$ & $(2\cdot8,0;[(1,2),(1,2),2],[(1,2),(5,8),8],[(0,1),(7,8),8])$ & $[(2,0;1;);(8,0;((1,4),1),((7,8),2))]$  \\
				\hline
$24$ & $(2\cdot8,0;[(1,2),(1,2),2],[(1,2),(7,8),8],[(0,1),(5,8),8])$ & $[(2,0;1;);(8,0;((3,4),1),((5,8),2))]$ \\
				\hline
$25$ & $(4\cdot4,0;[(1,4),(0,1),4],[(3,4),(1,4),4],[(0,1),(3,4),4])$ & $[(4,0;((1,4),4));(4,0;((3,4),4))]$  \\
				\hline
$26$ & $(4\cdot4,0;[(3,4),(0,1),4],[(1,4),(3,4),4],[(0,1),(1,4),4])$ & $[(4,0;((3,4),4));(4,0;((1,4),4))]$  \\
				\hline
$27$ & $(4\cdot4,0;[(1,4),(0,1),4],[(3,4),(3,4),4],[(0,1),(1,4),4])$ & $[(4,0;((1,4),4));(4,0;((1,4),4))]$  \\
				\hline
$28$ & $(4\cdot4,0;[(3,4),(0,1),4],[(1,4),(1,4),4],[(0,1),(3,4),4])$ & $[(4,0;((3,4),4));(4,0;((3,4),4))]$  \\
				\hline
$29$ & $(4\cdot4,0;[(3,4),(1,2),4],[(1,4),(1,4),4],[(0,1),(1,4),4])$ & $[(4,1;((1,2),2));(4,0;((1,4),4))]$  \\
				\hline
$30$ & $(4\cdot4,0;[(3,4),(1,2),4],[(1,4),(3,4),4],[(0,1),(3,4),4])$ & $[(4,1;((1,2),2));(4,0;((3,4),4))]$  \\
				\hline
$31$ & $(4\cdot4,0;[(1,4),(1,2),4],[(3,4),(1,4),4],[(0,1),(1,4),4])$ & $[(4,1;((1,2),2));(4,0;((1,4),4))]$ \\
				\hline
$32$ & $(4\cdot4,0;[(1,4),(1,2),4],[(3,4),(3,4),4],[(0,1),(3,4),4])$ & $[(4,1;((1,2),2));(4,0;((3,4),4))]$  \\
				\hline	
			\end{tabular}}
			\caption{The weak conjugacy classes of two-generator finite abelian subgroups of $\Mod(S_3)$. (*The suffix refers to the multiplicity of the tuple in the abelian data set.)} 
       \label{tab:wkc_mods3}	
	\end{table}

Note that the actions S.Nos 17-24 in Table~\ref{tab:wkc_mods3} have irreducible Type 1 actions as one of their generators. As the structure realizing such an action is unique, by lemma~\ref{lem:fix_abelian}, the abelian groups representing these weak conjugacy classes are realized as isometry groups by a unique structure. 

\section*{Acknowledgements}
The first author was supported by a UGC-JRF fellowship. The authors would like to thank Dheeraj Kulkarni and Siddhartha Sarkar for some helpful discussions.

\bibliographystyle{plain}
\bibliography{ab_real}
\end{document}